\documentclass{article}
\usepackage{color,bbold}%, hyperref
\usepackage{amssymb,amsmath,amsthm,graphicx, color, amsfonts,tabularx}
\usepackage[english]{babel}
\usepackage{afterpage}

\usepackage{authblk, relsize}
\usepackage{float}%insert images in the exact place
\usepackage{geometry}
\def\supp{\mathop{\rm supp}\nolimits}

\newtheorem{theorem}{Theorem}[section]

\newtheorem{lemma}[theorem]{Lemma}
\newtheorem{proposition}[theorem]{Proposition}

\newtheorem{definition}[theorem]{Definition}
\newtheorem{remark}[theorem]{Remark}
\newtheorem{example}[theorem]{Example}

\newcommand{\N}{\mathbb{N}}

\newcommand{\R}{\mathbb{R}}

\newcommand{\bchi}{\mathlarger{\chi}}
\renewenvironment{proof}[1][.]{%
\bigskip\noindent{\bf Proof#1 }}{%
\hfill$\blacksquare$\bigskip}
\newcommand{\mS}{\mathcal S}
\newcommand{\mF}{\mathcal F}
\newcommand{\F}{\mathcal F}

\newcommand{\K}{\mathcal K}

\newcommand{\FK}{\mathcal{F}_\mathcal{K}}

\newcommand{\mpS}{\mathcal{S}_{\on{mp}}}
\newcommand{\fS}{\mathcal{S}_{\on{f}}}
\newcommand{\IMS}{M_\mathcal{S}}
\newcommand{\FMS}{Z_\mathcal{S}}

\newcommand{\ve}{\varepsilon}
\newcommand{\on}{\operatorname}
\newcommand{\vp}{\varphi}

\begin{document}
\pagestyle{myheadings}
\title{{Fuzzy-set approach to invariant idempotent measures}}
\author[1]{Rudnei D. da Cunha
\thanks{
Instituto de Matem\'atica e Estat\'istica - UFRGS, Av. Bento Gon\c calves 9500, 91500-900 Porto Alegre - RS - Brazil.\\
E-mail: rudnei.cunha@ufrgs.br}}
\author[1]{Elismar R. Oliveira
\thanks{E-mail: elismar.oliveira@ufrgs.br}}
\author[2]{Filip Strobin\thanks{
Institute of Mathematics, Lodz University of Technology, W\'olcza\'nska 215, 90-924 {\L}\'od\'z, Poland. \; E-mail: filip.strobin@p.lodz.pl}}
\affil[1]{Universidade Federal do Rio Grande do Sul}
\affil[2]{Lodz University of Technology}

\date{\today}
\maketitle

\noindent\rule{\textwidth}{0.4mm}
\begin{abstract}
We provide a {new approach to} the Hutchinson-Barnsley theory for idempotent measures first presented {in N. Mazurenko, M. Zarichnyi, \emph{Invariant idempotent measures,} Carpathian Math. Publ., 10 (2018), 1, 172--178.
} The main feature developed here is a metrization of the space of idempotent measures using the {embedding of the space of idempotent measures to the} space of fuzzy sets. The metric obtained induces a {topology} stronger than the {canonical} pointwise convergence {topology}. A key result is the existence of a bijection between idempotent measures and fuzzy sets and  a conjugation between the Markov operator {of an} IFS on idempotent measures and the fuzzy fractal operator {of the} associated Fuzzy IFS. {This allows} to prove that the Markov operator for idempotent measures is a contraction w.r.t. the induced metric and, from this, to obtain a convergence theorem and { algorithms} that  % \ero{Elismar: this part goes to another work:: to approximate integrals w.r.t. the invariant idempotent measures{, as well as}}
 draw pictures of invariant measures as greyscale images.

%{ algorithms} \ero{Elismar: this part goes to another work:: {We also show that the Markov operator for an IFS on idempotent measures is a contraction w.r.t. other metrics on $I(X)$, %  metric $d_1$ defined in
%A.A. Zaitov, {\em  On a metric of the space of idempotent probability measures.} Appl. Gen. Topol., 21 (2020), 1, 35--51.
%which gives alternative proofs of the existence of invariant idempotent measures.}}
\end{abstract}
\vspace {.8cm}
\noindent\rule{\textwidth}{0.4mm}
\emph{{Key words and phrases: idempotent measures, iterated function systems, {atractors}, {fractals}, fuzzy sets, invariant measures, {algorithms generating fractal images}}}

\emph{2010 Mathematics Subject Classification: Primary {28A80, 28A33, 37M25; Secondary 37C70, 54E35, 65S05}}

\section{Introduction}
Its been well established in the literature the use of Radon probability measures {in studies of} many systems{, such} as random dynamical systems, Markov chains and Iterated Function Systems (IFS for short) among many others.  In particular, the Hutchinson-Barnsley theory \cite{HUT, BAR88} has settled, in the 80's, the modern basis to {study IFSs and} connecting the attractor to the invariant probability measure which has support on it. From this theory one can get {outstanding} results such as the {Chaos Game} \cite{BAR88}, allowing to draw the attractor by picking iterations according to a probability, or the Ergodic Theorem \cite{Elt87} or \cite{BDEG88}, allowing to compute integrals w.r.t the invariant measure through averages {of randomly picked} iterations according to a probability.\\
The idempotent analysis (a part of the idempotent mathematics),
founded by  Maslov and his collaborators \cite{Lit02} and \cite{Lit07}, brought the notion of {an} idempotent (or Maslov) measure %{ algorithms}
%\ero{Elismar: I know that it is the name used in the Maslov work, but I was wondering if we should make some adjustment to differentiate measures (as positive signed measures) from probability measures (positive measures with mass equal to 1, along all the text??}{\color{blue}FILIP: For me not necessarily, but we can make appropriate changes if you wish. BTW, I erased the word $"$algorithms$"$ just before this sentence, is it ok?}
with important applications in different parts of mathematics, such as optimization. Roughly speaking, it is a non additive integration theory built over a max-plus semiring.  The natural question arising in this setting is the existence of an idempotent version of the Hutchinson-Barnsley {theory}.  {For example, is it possible to associate in {``}reasonable way" an invariant idempotent measure to each IFS having an attractor?}  The answer came in a natural way defining the {analogon} of the Markov operator acting on idempotent measures, where each map of {an} IFS act on a measure by a pushforward {transformation}, and finding its fixed points which {should} be the invariant idempotent measures.

In \cite{MZ} it was proved that each max-plus normalized IFS ${\mpS}$ generates the unique invariant idempotent measure $\mu_{\mS}$, which a contractive fixed point (w.r.t. the canonical {pointwise convergence} topology $\tau_p$ on $I(X)$) of the appropriately defined idempotent Markov operator ${\IMS}$. The proof is topological and does not base on the possible contractiveness of $\IMS$ w.r.t. some metric on $I(X)$. In the paper we give an alternative proof of this fact by using a completely different approach.

We first define a natural bijection $\Theta$ between the space $I(X)$ and the space $\FK(X)$ of {``}compact{"} fuzzy subsets of $X$, so that the topology induced on $I(X)$ by $\Theta$ is stronger than $\tau_p$. Then we prove that each max-plus normalized IFS ${\mpS}$ on $I(X)$ generates appropriate fuzzy IFS ${\fS}$ on ${\FK(X)}$ and such that the idempotent Markov operator ${\IMS}$ and the fuzzy Markov operator ${\FMS}$ generated by ${\fS}$ are conjugated w.r.t. the map $\Theta$, which means that ${\Theta\circ \IMS=\FMS\circ\Theta}$. This result can move the whole discussion to the setting of fuzzy attractors of fuzzy IFSs. In particular, the existence of idempotent invariant measure $\mu_{\mathcal{S}}$ for contractive max-plus IFS can be explained by the existence of fuzzy attractor ${u_{\mS}}$ of the corresponding fuzzy IFS ${\fS}$.

Moreover, this approach allows to define a natural metric $d_\theta$ on $I(X)$ so that idempotent Markov operators generated by contractive max-plus normalized IFSs are contractive, and hence the existence of idempotent invariant measure for such IFSs can be explained by the classical fixed point theorems.

{Then} we imply certain algorithms for fuzzy IFSs to get the approximations and illustrations of the idempotent invariant measures.

%\ero{Elismar: Moved from here to other work?  {Finally, to complete the picture of contractivity approach to max-plus normalized IFSs, we show that the idempotent Markov operators of contractive max-plus normalized IFSs are also contractive w.r.t. the metric $d_1$ on $I(X)$ defined in \cite{ZAI}. This gives another proof of the existence of idempotent measures.}}

%\section{The semiring $S$}\label{sec:semi_ring_r_max}

\section{{The space of idempotent probability measures}}\label{sec:set_prob_measures_I(X)}
Let {$\mathbb{R}_{{\rm max}}:=\mathbb{R} \cup \{-\infty\}$} {be the extended set of real numbers}. Consider the operations
\begin{itemize}
  \item $\oplus: \mathbb{R}_{{\rm max}}\times \mathbb{R}_{{\rm max}} \to \mathbb{R}_{{\rm max}}$ given by \\ $x \oplus y = \max \{ x, y\}$ and
  \item $\odot: \mathbb{R}_{{\rm max}}\times \mathbb{R}_{{\rm max}} \to \mathbb{R}_{{\rm max}}$ given by \\ $x \odot y = x+y$.
\end{itemize}
Then we define the \emph{max-plus semiring} $S$ as the algebraic  structure $S=(\mathbb{R}_{{\rm max}}, \oplus, \odot)$. From the usual conventions we obtain
\begin{itemize}
  \item $-\infty \oplus x= x  \oplus -\infty= x$, for all $x \in \mathbb{R}_{{\rm max}}$ and
  \item $0 \odot x= x  \odot 0= x$, for all $x \in \mathbb{R}_{{\rm max}}$.
\end{itemize}
Then $S$ is a semiring with null element  $\mathbb{0}=-\infty$ and identity element $\mathbb{1}=0$. Moreover, $S$ is idempotent because
{
\begin{itemize}
\item $x \oplus x= x $, for all $x \in \mathbb{R}_{{\rm max}}$.
%\ero{Elismar: perhaps it shouldn't be itemized?}{\color{blue}FILIP: ok, we can change it, but on the other hand there is kind of logic - we itemize all properties}
\end{itemize}
}
Let $X$ be a compact Hausdorff space. We consider the usual algebra of continuous functions from $X$ to $\mathbb{R}$ denoted { by $C(X)$.}

Now we introduce {the certain} algebraic structure in $C(X)$ by{defining} the operations $\oplus$ and $\odot${, for any $\varphi, \psi \in C(X)$}:
\begin{itemize}
  \item $(\varphi \oplus \psi)(x)= \varphi(x)  \oplus \psi(x)$, for all $x \in X$ and
  \item $(\varphi \odot \psi)(x)= \varphi(x)  \odot \psi(x)$, for all $x \in X$.
\end{itemize}

\begin{definition}\label{def:idemp_prob_measure}\cite{Zar, ZAI}
  \emph{A functional (not necessarily linear nor continuous) $\mu:C(X) \to \mathbb{R}$ satisfying
  \begin{enumerate}
    \item $\mu(\lambda)=\lambda$ for all $\lambda\in\R$ (normalization);
    \item $\mu(\lambda \odot \psi)=\lambda \odot \mu(\psi)$, for all $\lambda \in\mathbb{R}$ and $\psi \in C(X)$;
    \item $\mu(\varphi \oplus \psi)=\mu(\varphi) \oplus \mu(\psi)$, for all $\varphi, \psi \in C(X)$,
  \end{enumerate}
  is called an} idempotent probability measure \emph{(or }Maslov measure\emph{)}.
\end{definition}

{Note that in appropriate places above we identified real values $\lambda$ with constant maps defined on $X$}.
The set of all idempotent probability measures in $X$ is denoted $I(X)$.% and by analogy, we wrote
%$$\mu(\psi)=\int_{X}\psi d\mu,$$
%for all $\psi \in C(X)$.

%Another consequence of the definition is  that, if $\varphi \leq \psi$ then $\mu(\varphi) \leq \mu(\psi)$. {[FILIP: Do we need this property? If not, maybe we can erase it?]}\ero{Elismar: Yes for both questions, we can erase all of this. Although we use the notation $\mu(\psi)$ in Lemma~\ref{fil5}, but it shouldn't be a problem.}

Canonically (see \cite{Rep10}, \cite{Zar} and \cite{ZAI}), we endow $I(X)$ with the topology $\tau_p$ of the pointwise convergence, that is, the basis of the topology $\tau_p$ consists of sets of the form
$$
\{\mu\in I(X):|\mu(\varphi_1)-\mu_1(\varphi_1)|<\varepsilon,...,|\mu(\vp_n)-\mu_n(\vp_n)|<\varepsilon\}{,}
$$
where $\varepsilon>0$, $n\in\N$ and $\mu_1,...,\mu_n\in I(X)$ and $\vp_1,...,\vp_n\in C(X)$. Equivalently, $\tau_p$ is the topology on $C(X)$ induced from the Tychonoff product topology on $\R^X$.\\
Clearly, for $(\mu_n)\subset I(X)$ and $\mu\in I(X)$, we have that
$\mu_n \overset{\tau_p}{\to} \mu$ in $I(X)$ {if and only if} $\mu_n(\psi) \to \mu(\psi)$ for all $\psi \in C(X)$.

\begin{example}
  Consider $x_0 \in X$ a fixed point and the functional $\mu:C(X) \to \mathbb{R}$ given by $\mu(\psi):=\psi(x_0),$ for all $\psi \in C(X)$. Obviously $\mu(0)=0(x_0)=0$ and $$\mu(\lambda \odot \psi)= (\lambda + \psi)(x_0)=\lambda + (\psi(x_0))=\lambda \odot \mu(\psi).$$ For the last, $$\mu(\varphi \oplus \psi)=(\vp\oplus\psi)(x_0)=\varphi(x_0) \oplus \psi(x_0)  =\mu(\varphi) \oplus \mu(\psi)$$ shows that $\mu \in I(X)$. This is the Dirac idempotent probability measure with base point $x_0 \in X$ denoted $\mu= \delta_{x_0}$.
\end{example}

\begin{example}
  Consider { a sequence $(x_n)\subset X$ that is convergent to a point $x_0\in X$.} % $x_{n} \to x_{0} \in X$ a sequence convergent to a fixed point
  Obviously, for every $\psi\in C(X)$, we have $\delta_{x_{n}}(\psi)= \psi(x_{n}) \to \psi(x_{0})= \delta_{x_{0}}(\psi)$ because $\psi$ is continuous. Thus {$\delta_{x_{n}} \to \delta_{x_{0}}$.}
\end{example}

In order to define conditional probabilities and the support we consider a continuous map $f$ from another compact metric Hausdorff space $Y$ to $X$. There is a canonical way to relate $I(Y)$ to $I(X)$ via a covariant functor $I(f): I(Y) \to I(X)$ given by
$$ \forall_{\vp\in C(X)}\;I(f)(\nu) (\varphi)= \nu(\varphi\circ f),$$
for all $\nu\in I(X)$.

%{[FILIP: For me the definiton of support is not clear, I looked at the references and it probably cannot be simplified. What about defining support as later, with use of density, and just make a remark that in references another equivalent definition is also given? The adventage is that we will shorten the paper - the preparatory part is quite too long w.r.t. main content]} \ero{Elismar: Yes, I agree, it has been a point of concern of my own.}

%A particularly important case is the inclusion map. In this case we consider a closed subset $A \subseteq X$ and take $Y=A$ and $f(y)=i_{A}(y)=y, \forall y \in A$. Then we can see {that} $I(A)$ as a subset of $I(X)$ in the sense that $I(i_{A})(I(A))\subseteq I(X)$. In this case, given $\mu \in I(X)$ we say that $\mu \in I(A)$ or, that $\mu$ is restricted to $A$, if $ \mu \in I(i_{A})(I(A))$.

%Now we can define the \emph{support} of an idempotent probability measure $\mu \in I(X)$  as the set
%$$\operatorname{supp} \mu:=\{\cap A \; | \; {A\mbox{ is closed and }} \; \mu \in I(A)\}.$$

\begin{remark}\label{fil4}\emph{
In the papers \cite{Zar} and \cite{ZAI} there are considered the problem of metrization of $I(X)$. Despite natural counterparts of classical metrics on the space of probablility measures are only pseudometrics, the topology $\tau_p$ on $I(X)$ is metrizable. In particular, $\tau_p$ can be considered as the topology induced by the notion of convergence.} %Its definition of $d_I$ is complicated, but in a natural way there can be defined a metric $d_1$ on $I(X)$ so that $d_I\leq d_1$ and in particular, the topology induced by $d_1$ is finer than $\tau_p$. The metric $d_1$ is defined as follows: for $\mu_1,\mu_2\in I(X)$, set
%$$
%d_{1}\left(\mu_{1}, \mu_{2}\right)=\inf \left\{\sup \{d(x, y):(x, y) \in \operatorname{supp} \xi\}: \xi \in \Lambda_{12}\right\}
%$$ where
%$\Lambda_{12}$ is the family of all idempotent measures $\xi\in I(X\times X)$ with $I\left(\pi_{i}\right)(\xi)=\mu_{i}, i=1,2$, and where $\pi_i$, $i=1,2$, are natural projections of $X\times X$ onto $X$.\\
%Finally, let us note that since the topology $\tau_p$ is metrizable, it can be considered as the topology induced by the notion of %convergence $\mu_n\overset{\tau_p}{\to}\mu$.}
\end{remark}

The next key idea is the \emph{density} of an idempotent probability measure introduced in \cite{Kol88} and developed in \cite{Kol97} and \cite{Aki99}.  Given $\mu \in I(X)$ we can always define the density function of $\mu$,  $ \lambda_{\mu} : X \to [-\infty,\, 0]$  by
$$ \lambda_{\mu}(x)= \inf\{ \mu(\varphi) \, | \, \varphi \in C(X), \; \varphi \leq 0, \; \varphi(x)=0 \}.$$
The following lemma lists basic properties of densities see, e.g., \cite[page 39]{ZAI}, \cite[Theorem 1.5]{Kol97} (where the $-\infty$ is replaced by $\infty$ and thus the upper semicontinuity is replaced by the lower one; note that also in (\ref{fil1}) we have $\max$ since upper semicontinuous map on compact space attains its maximum)  \cite{MoDo99}.

%\ero{Elismar:I see, it is a point that makes me worried. But it seems to be a bit odd to require homomorphism w.r.t to the $\odot=+$, i.e, $\mu(\psi\odot \phi)=\mu(\psi)\odot\mu(\phi)$ because it leads to $\sigma$-additivity, or not? In this case $\mu$ will be also an usual probability? The only thing we really need is to have idempotent measures uniquely defined by densities which are usc functions. Without that we have no paper. So we should chose the appropriated setting where it true.  I believe that the Example 1 in MASLOV IDEMPOTENT PROBABILITY CALCULUS, I, Moral and Doisy is the appropriated setting, he is defining densities implicitly.\\ FILIP: Yes, probably this will be a good reference - I added it.}
\begin{lemma}\label{fil6}
$\;$
\begin{itemize}
\item[(1)] The density $\lambda_{\mu}$ of an idempotent measure $\mu\in I(X)$ has the following properties:
\begin{itemize}
\item[(1i)] $\lambda_{\mu}$ is upper semicontinuous {(usc for short)};
\item[(1ii)] $\lambda_{\mu}(x)=0$ for some $x\in X$;
\item[(1iii)] $
\mu=\bigoplus_{x\in X}\lambda_{\mu}(x)\odot\delta_x$, that is, for every $\vp\in C(X)$, we have
\begin{equation}\label{fil1}
\mu(\vp)=\bigoplus_{x\in X}\lambda_{\mu}(x)\odot\vp(x)=\max\{\lambda_{\mu}(x)+\vp(x):x\in X\}{;}\end{equation}
\item[(1iv)] the density $\lambda_{\mu}$ of $\mu\in I(X)$ is uniquely determined.
\end{itemize}
\item[(2)] If $\lambda:X\to[-\infty,0]$ is upper semicontinuous and $\lambda(x)=0$ for some $x\in X$, then the map
$$\mu_{\lambda}=\bigoplus_{x\in X}\lambda(x)\odot\delta_x$$
is an idempotent measure, that is, $\mu_{\lambda}\in I(X)$.
\end{itemize}
\end{lemma}
%{I erased the proof, according to added remark before the statement}\ero{Elismar:ok}
%\begin{proof}
%For the proof of $(1i)-(1iii)$ and $(2)$ see \cite{ZAI} (let us point out that in (\ref{fil1}), maximum of the set $\{\lambda_{\mu}(x)+\vp(x):x\in X\}$ exists as any upper semicontinuous map on compact space attains its supremum).
%We will just give a proof of (1iv). Assume that $\lambda_1$ and $\lambda_2$ are densities of $\mu$ such that for some $x_0\in X$, we have $\lambda_1(x_0)\neq \lambda_2(x_0)$. We can assume that $\lambda_1(x_0)<\lambda_2(x_0)$. Since $\lambda_1$ is usc, the set $\{x\in X:\lambda_1(x)<\lambda_2(x_0)\}$ is open, hence contains a closed ball $\overline{B}(x_0,R)$. Now set
%$$
%M:=\max\{\lambda_1(x):x\notin B(x_0,R)\}.
%$$
%Then, using the Tietze extension theorem, we can find a continuous map $\vp:X\to(-\infty,0]$ such that:
%\begin{itemize}
%\item[(a)] $\vp(x_0)=0$;
%\item[(b)] $\vp(x)+M<\lambda_2(x_0)$ for $x\notin B(x_0,R)$.
%\end{itemize}
%On the one hand, we have
%$$
%\mu(\vp)=\max\{\lambda_1(x)+\vp(x):x\in X\}=$$ $$=\max\{\max\{\lambda_1(x)+\vp(x):x\in \overline{B}(x_0,R)\},\max\%{\lambda_1(x)+\vp(x):x\notin B(x_0,R)\}\}
%$$
%$$
%<\max\{\lambda_2(x_0),\max\{M+\vp(x):x\notin B(x_0,R)\}\}=\lambda_2(x_0)
%$$
%and on the other,
%$$
%\mu(\vp)=\max\{\lambda_2(x)+\vp(x):x\in X\}\geq \lambda_2(x_0)+\vp(x_0)=\lambda_2(x_0)
%$$
%which is a contradiction. The result follows.
%\end{proof}

Using the above lemma, if we define the set of densities
$$U_{S}(X):=\Big\{ \lambda: X \to [-\infty,\, 0]\, | \, \lambda{\text{ is usc and }\lambda(x_0)=0\text{ for some }x_0 \in X} \Big\},$$
then we see that
$$I(X)= \left\{\bigoplus_{x \in X}  \lambda(x)\odot \delta_{x} \; | \; \lambda \in U_{S}(X)\right\}.$$

%{[FILIP: If we decide to define support here, then it can be like this:]\ero{Elismar: I vote to define it here.}

An important notion is the support of an idempotent measure (see, e.g., \cite{ZAI}, \cite{Zar}). With the use of the notion of density, we give here an equivalent definition (see \cite{ZAI}): by the \emph{support} of an idempotent measure $\mu=\bigoplus_{x \in X}  \lambda(x)\odot \delta_{x}\in X$, we will mean the set
$$
\operatorname{supp} \mu:=\overline{\left\{x \in X: \lambda(x)>-\infty\right\}}.
$$
%[FILIP: I added closure - for example any continuous increasing map $\lambda$ from $X:=[0,1]$ to $[-\infty,0]$ shows that the set $\{x\in X:\lambda(x)>-\infty\}$ may not be closed]
%}\ero{Elismar: OK}
% The notion of density allows to define the support of an idempotent measure.

%\begin{remark}
%Let $A$ be a closed subset of a compact Hausdorf space $X$. It is easy to check that $\nu \in I(A)$ if and only if $\left\{x \in X: \lambda_{\nu}(x)>-\infty\right\} \subset A.$ Hence,
%$\operatorname{supp} \nu=\overline{\left\{x \in X: \lambda_{\nu}(x)>-\infty\right\}}.$
%\end{remark}
{Take another compact metric space $Y$ and fix a continuous map $\phi: X \to Y$. We now define the max-plus pushforward map $I(\phi): I(X) \to I(Y)$ given by
$$I(\phi)(\mu)(\varphi):=\mu( \varphi\circ \phi), \; \forall \varphi \in C(Y),$$
for any $\mu \in I(X)$.}
The next lemma gives a natural description of density of $I(\phi)$ (alternative description for idempotent measures with finite support can be found in \cite[page 484]{Zar}).

\begin{lemma}\label{fil5}In the above frame, for every
$\mu=\bigoplus_{x\in X}\lambda(x)\odot\delta_x\in I(X)$, we have that
\begin{equation}\label{filip1}
I(\phi)(\mu)=\bigoplus_{{y\in Y}}\lambda_\phi(y)\odot\delta_y
\end{equation}
where {
$$
\lambda_\phi(y)=\max\{\lambda(x):x\in\phi^{-1}(y)\}
$$
and we additionally assume  $\max\varnothing:=-\infty$.

Additionally, $\lambda_\phi\in U_S({Y})$, so it is the density of $I(\phi)$.}
%\left\{\begin{array}{ccc};\\
%-\infty&\mbox{if}&y\notin \phi(X).\end{array}\right.
%$$
\end{lemma}
\begin{proof}
Take any $\vp\in C(X)$. Then %{[FILIP: I modified a bit the calculations]\ero{Elismar: OK}
$$
I(\phi)(\mu)(\vp)=\mu(\vp\circ\phi)=\max\{\lambda(x)+\vp(\phi(x)):x\in X\}=$$
%$$
%=\max\{\max\{\lambda(x)+\vp(\phi(x)):x\in\phi^{-1}(y)\}:y\in\phi(X)\}=$$
$$
=\max\{\lambda(x)+\vp(y):y\in\phi(X),\;x\in\phi^{-1}(y)\}=$$
 $$=
\max\{\max\{\lambda(x):x\in\phi^{-1}(y)\}+\vp(y):y\in\phi(X)\}=
$$
$$
=\max\{\lambda_\phi(y)+\vp(y):y\in\phi(X)\}=\max\{\lambda_\phi(y)+\vp(y):y\in {Y}\}
$$
and we get (\ref{filip1}). Now take $x_0\in X$ so that $\lambda(x_0)=1$. Then $\lambda_\phi(y_0)=1$ for $y_0=\phi(x_0)$. Finally, let $(y_n)$ be a sequence convergent to $y$ so that $\lambda_\phi(y_n)\geq \alpha$ for some $\alpha\in\R$. Then we can find a sequence $(x_n)\subset X$ such that $\lambda(x_n)\geq \alpha$ and $\phi(x_n)=y_n$. As $X$ is compact, we can find its convergent subsequence $(x_{n_k})$ to some $x_0\in X$, such that necessarily $\phi(x_0)=y$. Since $\lambda(x_{n_k})\geq \alpha$ and $\lambda$ is usc, it also holds $\lambda_\phi(y)\geq \lambda(x_0)\geq \alpha$. Hence $\lambda_\phi$ is usc.
\end{proof}

\section{Hyperspace of fuzzy sets}\label{sec:Fuzzy sets}
We {now} recall some basic facts on fuzzy sets. Let {$(X,d)$ be a metric or topological space}.
\begin{definition} \emph{We say that $u$ is }a fuzzy subset of $X$\emph{ if $u: X \to [0,1]$. The family of fuzzy subsets of $X$ is denoted by ${\mF(X)}$, that is}
$${\mF(X)}:=\{ u \; | \; u: X \to [0,1]\}.$$
\end{definition}

In this theory \emph{fuzzy set} means that each point $x$ has a grade of membership  $0\leq u(x)\leq 1$ in the set $u$. Here, $u(x)=0$ indicates that $x$ is not in $u$ and $u(x)=0.4$ indicates that $x$ is a member of $u$ with membership degree $0.4$.

\begin{definition}\label{grey level} \emph{Given $\alpha \in (0,1]$ and $u \in \mathcal{F}_{X}$, }the grey level \emph{or} $\alpha$-cut of $u$\emph{ is the set
$$[u]^{\alpha}:=\{x \in X \; | \; u(x)\geq \alpha \},$$
that is, the set of points where the grey level exceeds the threshold value $\alpha$.
For $\alpha=0$ we define
$$[u]^{0}:={\operatorname{supp}(u):=} \overline{\bigcup\{[u]^{\alpha} \; | \;  \alpha >0\}}=\overline{\{x\in X:u(x)>0\}}{.}$$}
\end{definition}

\begin{definition} \emph{A fuzzy set $u \in {\mF(X)}$ is}\\
a) a crisp set\emph{, if $u(x) \in \{0, 1\}$ for every $x\in X$. We identify it with the classic subset $U=\{ x \in X \; |\; u(x)=1 \}$. In this case, $u$ is the indicator function of $U$: $u(x)=\bchi_{U}(x)$};\\
b) normal, \emph{if there is $x \in X$ such that $u(x)=1$;\\
c) }{compactly} supported\emph{ if $[u]^0$ is compact.}
\end{definition}
{Clearly, if $X$ is compact, then all fuzzy sets are compactly supported.}
Actually, the family of subsets of $X$, denoted by $2^X$, can be identified as a subset of ${\mF(X)}$, using the injective map $\chi: 2^X \to  {\mF(X)}$ defined by
$\bchi(B)=\bchi_{B}(x),$
for any $B \in 2^X$.

Fuzzy sets can be induced by maps. In his pioneering work in the 1965 {Zadeh~\cite[p. 346]{Zad}},  introduced what we call \emph{The Extension Principle}, that is a kind of pushforward map between fuzzy subsets.  {It plays an important role in modern science of computation and has been generalized in several ways, first by  \cite{Ngu78} and more recently by \cite{Bzo13}, \cite{Ful14} and many others}.

\begin{definition}(Zadeh's Extension Principle) \emph{Given a map $T: X \to Y$, $ u \in {\mF(X)}$, we define new fuzzy set $T(u)\in {\mF(Y)}$ as follows:
{$$T(u)(y) :=\sup\{u(x):x\in T^{-1}(y)\}$$
where we additionally assume $\sup\varnothing:=0$.}}
\end{definition}
%{[FILIP: do we need the definition of the inverse $T^{-1}(v)$? Maybe we can erase it?]}\ero{Elismar: Yes, we can erase, I probably forgot this.}
For additional properties of maps and operations between  fuzzy sets see \cite{Zad}.\\

{Finally, define}
$$
{\FK(X)}:=\{u\in\mF_X:u\;\mbox{is usc, normal and compactly supported}\}.
$$
The family ${\FK(X)}$ can be considered as a counterpart of hyperspace ${\K(X)}$ of all nonempty and compact subsets of $X$. We endow it with the metric ${d_{\on{f}}}$ defined by:
$$\forall_{u,v\in{\FK(X)}}\;{d_{\on{f}}} (u,v) := \sup_{\alpha \in [0,1]} h([u]^{\alpha},[v]^{\alpha}),$$
where $h$ is the Hausdorff metric. Note that elements of ${\FK(X)}$ are compactly supported, so all of their $\alpha$-cuts are nonempty and compact, and hence ${d_{\on{f}}}$ is well defined. In fact, we have
\begin{theorem}[\cite{Cab92},\cite{Oli17}] \label{Fuzzy Space is Complete} The function  ${d_{\on{f}}} : {\FK(X)} \times {\FK(X)} \to \R$ is a metric and $({\FK(X)} , {d_{\on{f}}})$ is a complete [compact] metric space provided $(X,d)$ is complete [compact].
\end{theorem}
In fact, the definition of ${d_{\on{f}}}$ can be simplified:
\begin{lemma}[\cite{Oli17}] \label{f2} In the above frame,
$$\forall_{u,v\in \mF^*_X}\;{d_{\on{f}}}(u,v)=\sup_{\alpha\in(0,1]}h([u]^\alpha,[v]^\alpha).$$
\end{lemma}

\section{Iterated function systems and their fuzzy and idempotent counterparts}
\subsection{Iterated function systems and the Hutchinson-Barnsley theorem}
\begin{definition}\emph{
If $X$ is a metric space and $\phi_1,...,\phi_L:X\to X$ are continuous, then we call $\mS=(X,(\phi_j)_{j=1}^L)$ as} an iterated function system \emph{(IFS for short).\\
Each IFS $\mS=(X,(\phi_j)_{j=1}^L)$ generates the }Hutchinson--Barnsley\emph{ operator $F_\mS:\K(X)\to\K(X)$ defined by
$$
\forall_{K\in\K(X)}\;F_\mS(K):=\bigcup_{j=1}^L\phi_j(K).
$$
A set $A_\mS\in\K(X)$ is called }the attractor \emph{of the IFS $\mS$, if
$$
A_\mS=F_\mS(A_\mS)=\bigcup_{j=1}^L\phi_j(A_\mS)
$$
and for every $K\in\K(X)$, the sequence of iterations $F_\mS^{(n)}(K)\to A_\mS$ w.r.t. the Hausdorff metric.}
\end{definition}

\begin{definition}\emph{We say that $\phi:X\to X$, where $X$ is a metric space, is a }Banach contraction\emph{, if its Lipschitz constant $\on{Lip}(\phi)<1$.\\
We say that $\phi:X\to X$ is a }Matkowski contraction\emph{, if there exists a nondecreasing function $\vp:[0,\infty)\to[0,\infty)$ so that {$\vp^{(n)}(t)\to 0$ for $t>0$ and}
$$
\forall_{x,y\in X}\;d(\phi(x),\phi(y))\leq\vp(d(x,y)).
$$
In this case, the map $\varphi$ is called as a }witness for $\phi$.
\end{definition}
\begin{remark}
\emph{Clearly, each Banach contraction is a Matkowski contraction, whereas it is known that there exist Matkowski contractions which are not Banach contractions. The Matkowski fixed point theorem \cite{MTK} states that each Matkowski contraction $\phi$ on a complete metric space $X$ satisfies the thesis of the Banach fixed point theorem, that is, there exists the unique fixed point $x_*\in X$, which is the limit of every sequence of iterates $(\phi^{(n)}(x))$, $x\in X$. In fact, Matkowski fixed point theorem is one of the strongest extensions of the Banach theorem.
{Finally, note that if $X$ is compact, then Matkowski contractivity is equivalent to so-called Edelstein contractivity, that is, $\phi:X\to X$ is Matkowski contraction { if and only if}
$$
\forall_{x,y\in X,\;x\neq y}\;d(\phi(x),\phi(y))<d(x,y).
$$
} We refer the reader to the paper \cite{JJ07}, in which many contractive conditions are discussed and compared.}
\end{remark}

\begin{definition}\emph{
Let $\mS=(X,(\phi_j)_{j=1}^L)$ be an IFS. We say that $\mS$ is }Matkowski [Banach, \emph{respectively}] contractive,\emph{ if each map $\phi_j$ is a Matkowski contraction [Banach contraction, respectively].}
\end{definition}
\begin{theorem}
Assume that $\mS=(X,(\phi_j)_{j=1}^L)$ is a Matkowski contractive IFS on a complete metric space. Then $\mS$ generates the unique attractor $A_\mS$.
\end{theorem}
The case when $\mS$ is Banach contractive is the statement of the the classical Hutchinson--Barnsley theorem \cite{BAR88}, \cite{HUT}. The general version for Matkowski contractive IFSs is also known (see, a.e., \cite{Oli17}, but also many other papers on extensions of Hutchinson-Barnsley theory) - it can be proved in a similar way as the classical version, as the Hutchinson operator $F_\mS$ turns to be a Matkowski contraction provided that $\mS$ is Matkowski contractive.

\subsection{IFSs on fuzzy sets}\label{sec:IFSs on fuzzy sets}

We say that a system of {maps $(d_j)_{j=1}^{L}: [0,1]\to [0,1]$  is \emph{an admissible system of grey level maps}} if it satisfies all the conditions\\
a) each $d_j$ is nondecreasing; \\
b) each $d_j$ is right continuous;\\%\ero{Filip: let's leave it as it is; if we need this property, we will just write it.}
c) for each $j$, we have that $d_j(0)=0$;\\
d) $d_j(1)=1$ for some $j$.

\begin{definition} \label{IFZS definition}
{\emph{If $\mS=(X,(\phi_j)_{j=1}^L)$ is an IFS  and $(d_j)_{j=1}^L$ is an admissible system of grey level maps, then we call the triple $\fS=(X,(\phi_j)_{j=1}^L, (d_{j})_{j=1}^{L})$ as a} fuzzy iterated function system %\ero{Elismar: Perhaps IFZS for short?, but I saw that we don't use abbreviations along of the manuscript} {\color{blue}[FILIP: Yes, we do not use it, so maybe lets leave it as it is? I feel that there shouldn't be too much abbreviations in the paper, it is complicated without them...}
\\}\emph{The operator $Z_{\mathcal{S}}:  {\FK(X) \to \FK(X)}$ defined by
$$Z_{\mathcal{S}}({u}):= \bigvee_{j \in \{1,\ldots,L\}} d_j\circ\phi_{j}(u){:=\max\{ d_j\circ\phi_{j}(u):j=1,\ldots,L\}}$$
is called }the fuzzy Hutchinson operator associated to ${\fS}$.\\
\emph{A fuzzy set $u_{\mS}\in {\FK(X)}$  is called }the fuzzy fractal attractor \emph{of ${\fS}$  if $Z_{\mathcal{S}}(u_{\mS})=u_{\mS}$, that is
$${u_{\mS}}= \bigvee_{j \in \{1,\ldots,L\}} d_j\circ\phi_{j}(u_{\mS}),$$
and for every $u\in{\FK(X)}$, the sequence of iterates $(Z^{(n)}_{\mathcal{S}}(u))$ converges to $u_{\mathcal{S}}$ with respect to the metric ${d_{\on{f}}}$.\\
We say that a fuzzy IFS }${\fS}$ is Matkowski [Banach, \emph{respectively}] contractive\emph{, if the underlying IFS ${\mS}$ is Matkowski {[Banach, respectively]} contractive.}
\end{definition}

The following result is a consequence of \cite[{Thmeorem} 3.15]{Oli17} (see {also }\cite{Dia94} for a more restrictive version) and can be considered as a fuzzy version of the Hutchinson--Barnsley theorem.

\begin{theorem}\label{lem3ff}
Let $(X,d)$ be a complete metric space and ${\fS}=(X,(\phi_j)_{j=1}^L, (d_j)_{j=1}^{L})$ be a Matkowski contractive IFZS. Then ${\fS}$ generates the unique fuzzy attractor $u_{\mS}$, {whose support equals $A_\mS$, the attractor of the underlying IFS $\mS$.}\\
In fact, the fuzzy Hutchinson operator $Z_\mS:{\FK(X)\to\FK(X)}$ is a Matkowski contraction {with} a witness $\varphi_\mS:=\max\{\varphi_j:j=1,...,L\}$, where $\varphi_j$s are witnesses for $\phi_j${, $j=1,...,L$}. {In particular, $Z_\mS$ is Banach contraction provided that $\fS$ is Banach contractive and the Lipschitz constant $\on{Lip}(Z_\mS)\leq\max\{\on{Lip}(\phi_j):j=1,...,L\}$.}
\end{theorem}

\subsection{IFSs on idempotent probabilities}\label{sec:IFSs on idempotent probabilities}
%\ero{ I'm a little concerned about the time we took to reach the main subject. However we are mixing three or four subjects which are not well known. It will be harder to the reader to catch up without a proper introduction} {\color{blue}FILIP: I agree but I do not have idea how to shorten the introduction. On the other hand, the main part takes 10 pages, and introduction circa 7, so it is acceptable balance, I believe}\\
We say that {$(X,(\phi_j)_{j=1}^{L}, (q_j)_{j=1}^{L})$}  is {a \emph{weighted IFS}}, if $(X,(\phi_j)_{j=1}^L)$ is an {IFS and [FILIP: I erased the compactness assumption I think it is not needed]}  $q_j\in \R$ for all $j=1,...,L$. Those IFSs are studied in ergodic theory (see \cite{Elt87}, \cite{BDEG88},\cite{Fan99}, \cite{DJP06} and \cite{Lop09})  and the special cases where  $ 0\leq q_j \leq 1$ and $\sum_{j=1}^{L}q_j=1$ are called IFS with probabilities  and has been largely studied by several authors proving the existence of an invariant probability measure. In \cite{MZ} there was considered the following version in the context of idempotent measures.

%{[Filip: maybe we can shorten the above and just write the below definition?] [I believe we could, I'm just a little worried about a proper motivation to define IFSs in idempotent probabilities, not in an artificial way as \cite{MZ} did.][Filip: ok, we will think of it]}

\begin{definition}
\emph{{Let $\mS=(X,(\phi_j)_{j=1}^m)$ be an IFS and $(q_j)_{j=1}^L$ is a family of real numbers so that
\begin{itemize}
  \item $ q_j \leq 0$ for $j=1,...,L$ {and, };
  \item $\displaystyle\bigoplus_{j=1,...,L} q_j =0$.
\end{itemize}
Then we call the triple $\mpS=(X,(\phi)_{j=1}^L,(q_j)_{j=1}^L)$ as}} a max-plus normalized IFS.

\emph{{Each max-plus normalized IFS $\mpS=(X,(\phi_j)_{j=1}^L,(q_j)_{j=1}^L)$} generates the map $M_{\mathcal{S}}:I(X)\to I(X)$, called as the} idempotent Markov operator, \emph{which adjust to every $\mu\in I(X)$, the idempotent measure $M_{\mathcal{S}}(\mu)$ defined by:
\begin{equation*}
M_{\mathcal{S}}(\mu):=\bigoplus_{j=1}^Lq_j\odot (I(\phi_j)(\mu))
\end{equation*}
that is, for every {$\psi\in C(X)$,}
\begin{equation*}
M_{\mathcal{S}}(\mu)({\psi})=\bigoplus_{j=1}^{L} q_{j}  \odot \mu({\psi}\circ\phi_j).
\end{equation*}
By  }an invariant idempotent measure \emph{of a {max-plus normalized IFS $\mpS$} we mean  the unique measure $\mu_{\mathcal{S}}\in I(X)$ which satisfies
\begin{equation*}
\mu_{\mathcal{S}}=M_{\mathcal{S}}(\mu_{\mathcal{S}})
\end{equation*}
and such that for every $\mu\in I(X)$, the sequence of iterates $M^{(n)}_{\mathcal{S}}(\mu)$ converges to $\mu_{\mathcal{S}}$ with respect to the $\tau_p$ topology on $I(X)$.\\
We say that a max-plus{ normalized IFS $\mpS$ is}} Matkowski [Banach, \emph{respectively}] contractive, \emph{if the underlying IFS ${\mS}$ is Matkowski [Banach, respectively] contractive.}
\end{definition}

The main result of \cite{MZ}, that is \cite[Theorem 1]{MZ}, states that:
\begin{theorem}
Each Banach contractive {max-plus normalized IFS $\mpS$} on a complete metric space generates the unique invariant idempotent measure $\mu_{\mathcal{S}}$.
\end{theorem}
As we mentioned in the introduction, the proof presented in \cite{MZ} is rather topological and does not use the fixed point theorem. Moreover, it is worth to note that despite it is stated for complete metric spaces, it is proved for compact spaces - in a simple way we can restrict the discussion from complete to compact spaces (see the beginning of \cite[Theorem 1]{MZ}). Thus also our results, { stated for compact spaces,} presented later{,} can be adjusted to complete spaces. Note that we obtain the thesis for Matkowski {contractive} contarctive {max-plus normalized} IFSs, so {it is} in fact an extension of the above theorem. On the other hand, it seems that the proof presented in \cite{MZ} can be rewritten for Matkowski contractive {IFSs}.

\section{Canonical bijections between idempotent measures and fuzzy sets and conjugation of idempotent Markov and fuzzy Hutchinon--Barnsley operators}\label{sec:The fuzzy correspondence principle}
{From now on, we assume that $(X,d)$ is a \textbf{compact} metric space}.\\
In this section we show that there is a natural correspondence between the space of idempotent measures $I(X)$ and the space ${\FK(X)}$  of fuzzy ``compact" sets (we assume that $X$ is a compact metric space).

Recall that {the set of densities is}
$$U_{S}(X)=\{ \lambda: X \to [-\infty,\, 0]\, | \, \lambda\text{ is usc and there exists some }x_0 \in X, \lambda(x_0)=0\},$$
{and the family of all idempotent measures is then}
$$I(X)= \left\{\bigoplus_{x \in X}  \lambda(x)\odot \delta_{x} \; | \; \lambda \in U_{S}(X)\right\}.$$

Hence {densities of idempotent probability measures} are very much like fuzzy sets {except the fact that their ranges are} $[-\infty, \, 0]$ instead $[0, \, 1]$.

In order to further investigate this analogy {we} {we call any increasing homeomorphism $\theta: [-\infty, \, 0] \to [0, \, 1]$ as a \emph{ scale function}.}
%\begin{enumerate}
%  \item $\theta(-\infty)=0$ and $\theta(0)=1$;
%  \item $\theta$ is strictly increasing;
%  \item $\theta$ is continuous in $[-\infty, \, 0]$.
%\end{enumerate}
{
\begin{example}
The maps $\theta_1(t)= \frac{1}{1+t^2}, \; t \in [-\infty, \, 0]$, as well as $\theta_2(t)= \frac{2}{\pi}\left(\frac{\pi}{2} + \arctan(t)\right), \; t \in [-\infty, \, 0]$ and $\theta_3(t)= a^t, \; t \in [-\infty, \, 0],\; a \in (1, +\infty)$ serve as examples of scale functions (the values of $\theta_i(-\infty)$ are appropriate limits at $-\infty$).
\end{example}}% The last one preserves better the aspect ratio.

{For any scale function $\theta$, define} the map $\Theta: I(X) \to {\mathcal{F}(X)}$
 by
  $$\Theta(\mu):= \theta\circ \lambda_{\mu} ,$$
for any $\mu \in I(X)$, where $ \lambda_{\mu} $ is the density of {$\mu$.}

\begin{lemma}\label{lem:theta_corresp}
The map {$\Theta$ is well defined and is a bijection between $I(X)$ and  $\FK(X)$}. Moreover, $\on{supp}(\mu)=\on{supp}(\Theta(\mu))$ for every $\mu\in I(X)$.
\end{lemma}
\begin{proof}
   {Fix any $\mu\in I(X)$. Then $\Theta(\mu)$ is well defined because the density is uniquely determined by idempotent probability measure. Since $X$ is a compact space, $\Theta(\mu)$ is compactly supported. Since $\lambda_\mu$ is usc and $\theta$ is continuous, $\Theta(\mu)$ is usc. Finally, since $ \lambda_{\mu} (x_0)=0$ for some $x_0\in X$, we also have that $\theta( \lambda_{\mu} (x_0))=\theta(0)=1$ meaning that  $\Theta(\mu)$ is normal}. Hence $\Theta:I(X)\to{\FK(X)}$.

To see that $\Theta$ is onto, take any $u\in{\FK(X)}$. Then $u=\theta\circ (\theta^{-1}\circ u)$, and $\theta^{-1}\circ u$ is usc and $\theta^{-1}\circ u(x_0)=0$ for $x_0$ with $u(x_0)=1$. Hence $\theta^{-1}\circ u$ is the density of some idempotent measure $\mu\in I(X)$ and $u=\Theta(\mu)$.\\
Recall that ${{\operatorname{supp}(u)}:=[u]^0:=}\overline{\{x\in X:u(x)>0\}}.$ Taking $u=\theta\circ\lambda_{\mu},$ we can see that $u(x)>0$ if, and only if,  $ \lambda_{\mu} (x)> -\infty$. Therefore, $\operatorname{supp}(u)= \operatorname{supp}(\mu)$.
\end{proof}
\begin{definition}\emph{
Assume that ${\mpS}=(X,(\phi_j)_{j=1}^L,(q_j)_{j=1}^L)$ is a max-plus normalized IFS. For every $j=1,...,L$, let $d_j:[0,1]\to[0,1]$ be defined by
$$
\forall_{t\in[0,1]}\;d_j(t):=\theta(q_j+\theta^{-1}(t)){.}
$$
Then the fuzzy IFS ${\fS}:=(X,(\phi_j)_{j=1}^L,d_j)_{j=1}^L)$ is called as }the corresponding {fuzzy IFS} for ${\mpS}$.
\end{definition}
Note that the definition of the corresponding {fuzzy IFS} is correct - by definition of $\theta$ and the assumptions on weights $(q_i)$, we see that $(d_j)$ is an admissible system of grey level maps (in fact, maps $d_j$ are even continuous).\\
The next result is crucial for our approach. It shows that idempotent Markov operator of ${\mpS}$ and the fuzzy Hutchinson--Barnsley operator of ${\fS}$ are conjugated.
\begin{theorem}\label{fil7}
Let {$\mpS$ be a max-plus normalized IFS and $\fS$} be its fuzzy correspondence. Then $M_{\mathcal{S}}$ and ${Z_{\mS}}$ are conjugated {via $\Theta$}, that is,
$$
\Theta\circ M_\mS={Z_{\mS}}\circ\Theta
$$
\end{theorem}
We precede the proof by the following lemma:
\begin{lemma}\label{fil23}
For $\mu=\bigoplus_{x\in X}\lambda(x)\odot\delta_x\in I(X)$, we have that
$$
M_{\mathcal{S}}(\mu)=\bigoplus_{y\in X}\lambda_\mS(y)\odot\delta_y
$$
where{
$$
\lambda_\mS(y)=\max\{q_j+\lambda_{\phi_j}(y):j=1,...,L\}=\max\{q_j+\lambda(x):j=1,...,L,\;x\in\phi^{-1}_j(y)\}
$$
%d_\mS(y)=\left\{\begin{array}{ccc}\max\{q_j+\lambda(x):j=1,...,L,\;x\in\phi^{-1}_j(y)\}&\mbox{if}&y\in\bigcup_{j=1}^L\phi_j(X)\\
%-\infty&\mbox{if}&y\notin\bigcup_{j=1}^L\phi_j(X)\end{array}\right.
%$$
and we additionally assume $\max\varnothing=-\infty$.

Additionally, $\lambda_\mS\in U_S(X)$, so it is the density of $M_\mS$.}
\end{lemma}
\begin{proof}%{[FILIP: I simplified the proof]}
We will use Lemma \ref{fil5}. {For any $\vp\in C(X)$, we have}
$$
M_{\mathcal{S}}(\mu)(\vp)=\bigoplus_{j=1}^Lq_j\odot (I(\phi_j)(\mu)(\vp))=\bigoplus_{j=1}^Lq_j\odot\Big(\bigoplus_{y\in {X}}\lambda_{\phi_j}(y)\odot \vp(y)\Big)=
$${
$$
\max\{\max\{q_j+\lambda_{\phi_j}(y):j=1,...,L\}+\vp(y):y\in X\}=\bigoplus_{y\in X}(\lambda_\mS(y)\odot \phi(y)){.}
$$
The fact that $\lambda_{\mS}\in U_S(X)$ follows from the fact that each $\lambda_{\phi_j}\in U_S(X)$.}
\end{proof}\\
\begin{proof}(of Theorem \ref{fil7})\\
Take any $\mu=\bigoplus_{x\in X}\lambda(x)\odot\delta_x$ and $\vp\in C(X)$. In view of earlier lemma, we have
$$
(\Theta\circ M_\mS)(\mu)=\Theta(M_\mS(\mu))=\theta\circ \lambda_\mS
$$
so for any $y\in {\bigcup_{j=1}^L\phi_j(X)}$, we have by continuity and monotonicity of $\theta$, that (recall that $\theta(-\infty)=0$) %{[FILIP: I also simplified the proof a bit]}\ero{Elismar: OK}
$$
(\Theta\circ M_\mS)(\mu)(y)=\theta(\lambda_\mS(y))=\theta(\max\{q_j+\lambda(x):j=1,...,L,\;x\in\phi^{-1}_j(y)\})=
$${
$$
=\max\{\theta(q_j+\lambda(x)):j=1,...,L,\;\mbox{with}\;x\in\phi_j^{-1}(y)\}=
$$
$$
=\max\{\theta(q_j+\theta^{-1}(\theta(\lambda(x))):j=1,...,L,\;\mbox{with}\;\phi_j^{-1}(y)\neq\varnothing\mbox{ and }x\in\phi_j^{-1}(y)\}=
$$
$$
=\max\{d_j(\theta(\lambda(x))):j=1,...,L,\;\mbox{with}\;\phi_j^{-1}(y)\neq\varnothing\mbox{ and }x\in\phi_j^{-1}(y)\}=
$$
$$
=\max\{d_j(\max\{(\theta\circ\lambda)(x):x\in \phi_j^{-1}(y)\}):j=1,...,L,\;\mbox{with}\;\phi_j^{-1}(y)\neq\varnothing\}=
$$
$$
=\max\{(d_j\circ \phi_j(\theta\circ\lambda))(y):j=1,...,L,\;\mbox{with}\;
\phi_j^{-1}(y)\neq\varnothing\}
=Z_\mS(\theta\circ\lambda)(y)=Z_\mS\circ\Theta(\mu)(y){.}
$$
If $y\notin\bigcup_{j=1}^L\phi_j(X)$, then
$$(\Theta\circ M_\mS)(\mu)(y)=\theta(\lambda_\mS(y))=0=\max\{d_j(0):j=1,...,L\}=$$ $$=\max\{d_j( \phi_j(\theta\circ\lambda)(y)):j=1,...,L\}=(Z_\mS\circ\Theta)(\mu)(y){.}$$
Since we considered all $y\in X$, we have}
$$
(\Theta\circ M_\mS)(\mu)=(Z_\mS\circ\Theta)(\mu).
$$
Finally, since $\mu$ was taken arbitrarily, we arrive to the assertion $\Theta\circ M_\mS=Z_\mS\circ\Theta.$
\end{proof}

\begin{remark}\emph{
   Consider any $\mu \in I(X)$, {continuous} $\phi:X \to X$ and $u=\Theta(\mu) \in{\FK(X)}$. Then {using Theorem \ref{fil7} for $L=1$ and $q_1=0$, we see that
   $\Theta(I(\phi)(\mu))=\phi(\Theta(\mu))$}
   meaning that the correspondence $\Theta$ conjugates the Zadeh's extension principle and the max-plus pushforward { operator $I$.}}
\end{remark}

%we assume that ${\mathcal{S}=(X,(\phi_j)_{j=1}^L},(q_j)_{j=1}^L)$ is a {[Matkowski]} contractive nam-plus IFS, where $\mS=(X,(\phi_j)_{j=1}^L)$ and $X$ is compact, and that ${\mS_F=(X,(\phi_j)},(d_j)_{j=1}^L)$ is its fuzzy correspondence.
Now we investigate the properties of {the topology $\tau_\theta$} induced from {$\FK(X)$} via the map $\Theta$. Let ${d_{\theta}}$ be the metric on $I(X)$ defined by
$$d_{\theta}(\mu, \nu)= {d_{\on{f}}}(\Theta(\mu), \Theta(\nu)).$$
Clearly, $d_\theta$ is a metric (recall that $\Theta$ is bijection) and the spaces $(I(X),d_\theta)$ and $({\FK(X)},{d_{\on{f}}})$ are homeomorphic {and $\Theta$ is a homeomorphism}.

{ We point out that there exists other approaches on the literature where a {fuzzy metric} is introduced on $I(X)$ {which gives a way to introduce also a metric, see e.g., \cite{BSZ}}.  {It seems that this approach} does not provide any underlying structure useful to study IFSs and its operator. This is the major advantage of our approach identifying $I(X)$ with the metric space of fuzzy sets.}

{First we show that the metric $d_\theta$ does not depend on the choice of a scaling map. We skip a straightforward proof.
\begin{lemma}\label{filipmetric}
For every { $\mu=\bigoplus_{x\in X}\lambda(x)\odot\delta_x,\;\nu=\bigoplus_{x\in X}\eta(x)\odot\delta_x\in I(X)$}, we have
$$
d_\theta(\mu,\nu)=\sup_{\beta\in(-\infty,0]}h(\{x\in X:\lambda(x)\geq \beta\},\{x\in X:\eta(x)\geq \beta\}){.}
$$
\end{lemma}
%\begin{proof}
%We have
%$$
%d_{\theta_1}(\mu,\nu)=\sup_{\alpha\in(0,1]}h([\theta_1\circ \lambda]^\alpha,[\theta_1\circ \eta]^\alpha)=
%$$
%$$
%=\sup_{\alpha\in(0,1]}h(\{x\in X:\theta_1(\lambda(x))\geq \alpha\},\{x\in X:\theta_1(\eta(x))\geq %\alpha\})=\sup_{\alpha\in(0,1]}h(\{x\in X:\lambda(x)\geq \theta_1^{-1}(\alpha)\},\{x\in X:\eta(x)\geq \theta_1^{-1}(\alpha)\})=
%$$
%$$
%=\sup_{\beta\in(-\infty,1]}h(\{x\in X:\lambda(x)\geq \beta\},\{x\in X:\eta(x)\geq \beta\})=
%$$
%\end{proof}

}

\begin{proposition}\label{fil11}
The {metric space $(I(X),d_\theta))$ is compact and the} topology induced by $d_\theta$ is finer than the topology $\tau_p$. In other words, {$\tau_p\subset\tau_\theta$.}% for every $(\mu_n)\subset I(X)$ and $\mu\in I(X)$, if $\mu_n\overset{d_\theta}{\to}\mu$, then then $\mu_n\overset{\tau_p}{\to}\mu$ w.r.t. $\tau_p$.
\end{proposition}

\begin{proof}{The metric space $(I(X),d_\theta))$ is compact as the metric $d_{\on{f}}$ is compact (recall that we assume that $(X,d)$ is compact).}
Since $(I(X),\tau_p)$ is metrizable, it is enough to show that, in fact, the convergence w.r.t. ${d_{\theta}}$ implies the convergence w.r.t. $\tau_p$. Hence let $\mu_n\to \mu$ w.r.t. ${d_{\theta}}$. It is sufficient to show that for any $\vp\in C(X)$, the sequence $(\mu_n(\vp))$ converges to $\mu(\vp)$.\\
Fix any $\varepsilon>0$ and take $\delta>0$ such that if $d(x,y)<\delta$ then $|\vp(x)-\vp(y)|<\varepsilon$. By our assumptions {and Lemma \ref{filipmetric}}, we can find $n_0\in\N$ such that for $n\geq n_0$ and {$\beta\in(-\infty,0]$, we have
\begin{equation}\label{filip2}
h(\{x\in X: \lambda_n (x)\geq \beta\},\{x\in X:\lambda(x)\geq \beta\})<\delta.
\end{equation}
}
where $\lambda_n ${, $n\in\N$,} are densities of $\mu_n$, and $\lambda$ is the density of $\mu$, respectively.\\
Fix $n\geq n_0$ {and choose} $x_1\in X$ such that
$$
\lambda_n (x_1)+\vp(x_1)=\sup\{\lambda_n (x)+\vp(x):x\in X\}=\mu_n(\vp).
$$
Setting $\beta_1:=\lambda_n (x_1)$, we see that $\beta_1>-\infty$, hence by (\ref{filip2}), we can find $x_1'\in X$ such that $d(x_1,x_1')<\delta$ and {${\lambda}(x_1')\geq \beta_1$.} %=\lambda_n (x_1)$$
 Moreover, as $d(x_1,x_1')<\delta$, we have that
{$
|\vp(x_1)-\vp(x_1')|<\varepsilon
$.}
All in all, we have
$$
\mu_n(\vp)=\lambda_n (x_1)+\vp(x_1)\leq {\lambda}(x_1')+\vp(x_1)-\vp(x_1')+\vp(x_1')\leq $$ $$\leq\lambda(x_1')+|\vp(x_1)-\vp(x_1')|+\vp(x_1')<
\sup\{\lambda(x)+\vp(x):x\in X\}+\varepsilon=\mu(\vp)+\varepsilon.
$$
Hence
{$
\mu_n(\vp)-\mu(\vp)<\varepsilon.
$}
In the same way we can show that $\mu(\vp)-\mu_n(\vp)<\varepsilon$ {and} we get
$$
|\mu_n(\vp)-\mu(\vp)|<\varepsilon
$$
and the result follows.
\end{proof}

A natural question arises if the metric ${d_{\theta}}$ induce the topology $\tau_p$. It turns out that it is not the case.
\begin{example}Let $X:=[0,1]$, and
for $n\in\N$, let $\mu_n$ be the idempotent measure whose density is defined by
$$
\lambda_n (x)=\left\{\begin{array}{ccc}-x&\mbox{for}&x\in\Big[0,\frac{1}{n}\Big]\\
-\frac{1}{n}&\mbox{for}&x\in\Big[\frac{1}{n},1\Big]\end{array}\right.
$$
and let $\mu$ be the idempotent measure whose density equals $\lambda(x):=0$ for all $x\in [0,1]$. Clearly, for every $\vp\in C([0,1])$, we have
$$
\mu_n(\vp)=\max\{\lambda_n (x)+\vp(x):x\in[0,1]\}\to{\max}\{\vp(x):x\in[0,1]\}=\mu(\vp)
$$
Which means that $\mu_n\to\mu$ w.r.t. $\tau_p$.\\
However, for any $\theta(-\frac{1}{n})<\alpha\leq 1$, we have that $\theta^{-1}(\alpha)>-\frac{1}{n}$ and
$$
{\{x\in[0,1]:\lambda_n (x)\geq \theta^{-1}(\alpha)\}\subset \Big[0,\frac{1}{n}\Big]}
$$
and
$$
{\{x\in[0,1]:{\lambda}(x)\geq \theta^{-1}(\alpha)\}=[0,1].}
$$
Hence for $n\geq 2${, by Lemma \ref{filipmetric},} we have
$$
{d_{\theta}(\mu_n,\mu)\geq h(\{x\in [0,1]:\lambda_n (x)\geq\theta^{-1}(\alpha)\},\{x\in[0,1]:\lambda(x)\geq\theta^{-1}(\alpha)\}) \geq \frac{1}{2}}
$$
so $(\mu_n)$ is not convergent to $\mu$ w.r.t. ${d_{\theta}}$.
\end{example}

\section{The existence and properties of idempotent invariant measures via compact fuzzy sets}
Now we can state the alternative proof of the existence of invariant idempotent measure, and give its description in terms of fuzzy attractors. We assume that $\theta$ and $\Theta$ have the same meaning as earlier {and that $(X,d)$ is a compact metric space}.
\begin{theorem}\label{main}
Assume that ${\mpS}=(X,(\phi_j)_{j=1}^L,(q_j)_{j=1}^L)$ is a Matkowski contractive max-plus {normalized IFS. Then $\mpS$ generates the unique idempotent invariant measure $\mu_{\mathcal{S}}$.}

 % ${\fS}=(X,(\phi_j),(d_j)_{j=1}^L)$ is its fuzzy correspondence.\\
{In fact, the idempotent Markov operator $M_{\mathcal{S}}$ is Matkowski contractive w.r.t. $d_\theta$ with witness $\vp_\mS:=\max\{\vp_j:j=1,...,L\}$, where $\vp_j$s are witnesses for $\phi_j$s.

In particular, if $\mpS$ is Banach contractive, then $M_\mS$ is Banach contractive w.r.t. $d_\theta$ and $\on{Lip}(M_\mS)\leq\max\{\on{Lip}(\phi_j):j=1,...,L\}$.}
\end{theorem}
\begin{proof}
Using Theorem \ref{fil7} and {Theorem} \ref{lem3ff}, we have for all $\mu_1,\mu_2\in I(X)$:
$$
d_\theta(M_\mS(\mu_1),M_\mS(\mu_2))={d_{\on{f}}}(\Theta(M_\mS(\mu_1)),\Theta(M_\mS(\mu_2)))=$$ $$={d_{\on{f}}}({Z_{\mS}}(\Theta(\mu_1)),{Z_{\mS}}(\Theta(\mu_2)))\leq \varphi_\mS({d_{\on{f}}}(\Theta(\mu_1),\Theta(\mu_2)))=\varphi_\mS(d_\theta(\mu_1,\mu_2))
$$
Hence $M_\mS$ is a Matkowski contraction with a witness $\vp_\mS$. In particular, $M_{\mathcal{S}}$ has the unique fixed point $\mu_{\mathcal{S}}\in I(X)$, and for every $\mu\in I(X)$, the sequence of iterations $({M^{(n)}_{\mathcal{S}}}(\mu))$ converges to $\mu_{\mathcal{S}}$ w.r.t. $d_\theta$. By Proposition \ref{fil11}, we have that $(M_{\mathcal{S}}(\mu))$ converges to $\mu_{\mathcal{S}}$ w.r.t. $\tau_p$, hence $\mu_{\mathcal{S}}$ is the idempotent invariant measure of $\mathcal{S}$.
\end{proof}

{The next theorem gives the description of the idempotent invariant measure in terms of fuzzy-setting. It follows directly from  Theorem \ref{lem3ff} and Theorem \ref{fil7}.}
\begin{theorem}\label{fil13}
{Let ${\mpS}$ be a max-plus normalized IFS and $\fS$ be its fuzzy correspondence. Then $\mpS$ generates invariant idempotent measure $\mu_\mS$ if and only if $\fS$ generates fuzzy attractor $u_\mS$, and if this holds, then
$$\mu_{\mathcal{S}}=\Theta^{-1}(u_{\mS})$$
i.e., the density of $\mu_\mS$ equals $\theta^{-1}\circ u_{\mS}$. Moreover, the support $\on{supp}(\mu_\mS)=A_\mS$, the attractor of the underlying IFS $\mS$.

Also, for every $n\in\N$,}
\begin{equation}\label{filipAA}
M^{(n)}_{\mS}=\Theta^{-1}\circ Z^{(n)}_\mS\circ\Theta{.}
\end{equation}
%converges to $\mu_\mS$ w.r.t. $d_\theta$ and, in consequence, w.r.t $\tau_p$.}% has the following properties:
%\begin{itemize}
%\item[(i)] the density of $\mu_{\mathcal{S}}$ equals $\theta^{-1}\circ u_{\mS_F}$;
%\item[(ii)] for any $\mu\in I(X)$, the sequence $(\mu_n)$ whose densities equals
%$$
%\theta^{-1}\circ Z_{\mS_F}^n\circ \theta( \lambda_{\mu} )
%$$
%where $ \lambda_{\mu} $ is the density of $\mu$,
%is convergent to $\mu_{\mathcal{S}}$;
%\end{itemize}
\end{theorem}
%\begin{proof}{Condition (\ref{filipAA}) follows directly from Theorem \ref{fil7} and observe that
%$$
%M_\mS(\Theta^{-1}(u_{\mS_F}))=\Theta^{-1}(Z_{\mS_F}(u_{\mS_F}))=\Theta^{-1}(u_{\mS_F}).
%$$
%Hence $\Theta^{-1}(u_{\mS_F})$ is a fixed point of $M_\mS$, so it must be the idempotent measure. The second part follows easily from Theorem \ref{fil7}.%Now take any $\mu\in I(X)$, and use again Theorem \ref{fil7} to see that
%$$
%\Theta^{-1}\circ Z^{(n)}_{\mS_F}\circ\Theta(\mu)=M_\mS(\mu)
%$$
%which converges to $\mu_\mS$ by Theorem \ref{main}.
%\end{proof}

{ Before we start with the applications we would like to point out some remarks.
\begin{remark}\label{remarknew}\emph{
  Our first remark is that {all our results for IFSs} on idempotent measures can be immediately {adopted for the setting of \emph{generalized IFSs} (GIFSs)}, introduced by R. Miculescu and A. Mihail in 2008 (see \cite{Mih08} {and many other papers}) because we already studied the analogous of the results we employ here for fuzzy GIFSs and their discretizations in \cite{Oli17} and \cite{COS21}. In the last section we give a bit more details and present appropriate algorithms and use them for getting images of some idempotent invariant measures for GIFSs.\\
  The second one is that the theory of idempotent measures fits perfectly to generalize the work of \cite{LTV} to {fuzzy GIFSs} because the original work using Radon measures seems to be impossible to generalize for GIFSs of degree bigger or equal to 2.}
\end{remark}

\section{Applications of algorithm for fuzzy IFSs to obtain images of idempotent invariant measures}
In this section we show that we can apply the above results and discretizing ideas from our previous papers \cite{COS20} and \cite{COS21} to obtain an algorithm generating approximation of idempotent measures of Banach contractive {IFSs. We first present a deterministic algorithm, then its discrete version motivated by mentioned papers \cite{COS20} and \cite{COS21}% We start with giving basic notions from \cite{COS21}.

\subsection{Deterministic algorithm for generating idempotent invariant measure}
Assume that $\mpS$ is a Matkowski contractive max-plus normalized IFS and $\mu\in I(X)$. According to Theorem \ref{main}, iterations $M_\mS^{(n)}(\mu)$ gives better and better approximations of the idempotent measure $\mu_\mS$. As the whole information of the idempotent measures give their densities, the presented algorithm will deliver densities of measures $M_\mS^{(n)}(\mu)$. Note that if the $\mpS$ is Banach contractive, then after $n$ iterations we obtain the approximation of the attractor $\mu_\mS$ with the resolution
$$\delta\leq\frac{\alpha_\mS^n}{1-\alpha_\mS}d_\theta(\mu,M_\mS(\mu)){.}$$
In fact, we automatically obtain also approximations of the attractor $A_\mS$ of the underlying IFS $\mS$ with the same resolution.
{\tt
\begin{tabbing}
aaa\=aaa\=aaa\=aaa\=aaa\=aaa\=aaa\=aaa\= \kill
     \> \texttt{{ { DeterminIFSIdempMeasureDraw}}(${\mS}$)}\\
     \> {\bf input}: \\
     \> \>  \> $K \subseteq {X}$, any finite and {nonempty} subset (a list of points in ${X}$).\\
     \> \>  \> $\nu$, any discrete idempotent measure such that $\supp (\nu)= K$. \\
     \> {\bf output:} \\
     \> \> \>  A bitmap representing an approximation of the attractor.\\
     \> \> \> A bitmap  image representing a discrete invariant idempotent measure,\\ \>\>\>
 with a gray scale color bar indicating\\ \>\>\> the mass of each pixel.\\
     \> {\bf Initialize}  $\displaystyle\mu:={\bigoplus_{x \in X}{-\infty\odot\delta_{x}}} $,  $W:=\varnothing$\\
     \> {\bf for n from 1 to N do}\\
     \>  \>{\bf for $\ell$ from 1 to $\operatorname{Card}({K})$ do}\\
     \>  \>  \>  \>{\bf for $j$ from $1$ to $L$ do}\\
     \>  \>  \>  \>  \>${x}:={K}[\ell]$\\
     \>  \>  \>  \>  \>${y}:={\phi}_{j}({x})$\\
     \>  \>  \>  \>  \>${W:=W\cup\{y\}}$\\
     \>  \>  \>  \>  \> {$\mu({y}):=\max\{\mu({y}) ,  q_{j} + \nu({x})\}$}\\
     \>  \>  \>  \>{\bf end do} \\
     \>  \>{\bf end do} \\
    % \> \> {$K:=W$}\\% and $W:=\varnothing$ \\
     \> \> $\nu:=\mu$, $\displaystyle\mu:=\bigoplus_{x \in X}-\infty\odot\delta_{x}$, {$K:=W$ and $W:=\varnothing$}\\
     \>{\bf end do}\\
     \>{\bf return: Print} {$K$} and $\nu$ \\
\end{tabbing}}
%[FILIP: I moved the later remark here]
\begin{remark}\emph{
  Actually, we could obtain our approximation of $\mu_{\mS}$ by first approximating $u_{\mS}$ via \texttt{FuzzyIFSDraw($\mS$)} and then {use} $\mu_{\mS}=\Theta^{-1}(u_{\mS})$. However the adapted  version we presented here, \texttt{{ { DeterminIFSIdempMeasureDraw}}($\mS$)}, is {much faster} because we do not need the discretization of the extension principle. This is why we rewrite the algorithm \texttt{FuzzyIFSDraw($\mS$)} as \texttt{{ DeterminIFSIdempMeasureDraw}}($\mS$). The reciprocal {of this approach} will not work because not all fuzzy {IFSs are associated} to {max-plus normalized IFSs}.}
\end{remark}
\begin{example}\label{discrete_Fern} %\ero{FILIP: Good!}
This first example is based on a very well-known fractal, the Barnsley Fern. It is generated by the max-plus normalized IFS $\mpS=(X,(\phi_j)_{j=1}^4,(q_j)_{j=1}^4)$ defined by:
\[\mpS:
\left\{
  \begin{array}{ll}
      \phi_1(x,y)   & = (0.856 x + 0.0414 y + 0.07, -0.0205 x+ 0.858 y + 0.147)\\
      \phi_2(x,y)   & = ( 0.244 x - 0.385 y + 0.393, 0.176 x + 0.224 y + 0.102)\\
      \phi_3(x,y)   & = ( -0.144 x + 0.39 y + 0.527, 0.181 x + 0.259 y - 0.014 ) \\
      \phi_4(x,y)   & =( 0.486, 0.031 x + 0.216 y + 0.05)\\
  \end{array}
\right.
\]
and $q_1= -11, q_2= -7, q_3= 0$ and $q_4=0$.\\
We initialize the algorithm with $\displaystyle\nu:=\delta_{(0.5,0.5)}$ meaning that $K=\{(0.5,0.5)\}$, after each iteration we get respectively 4, 16, 64, 256, 1024, 4096, 16384, 65536, 262144, 1048576 and 4194304 points in $K$, which are depicted in Figure~\ref{Barnsley Fern} in a gray scale where each point brightness represents its density.
 \begin{figure}[H]
  \centering
  \includegraphics[width=9cm]{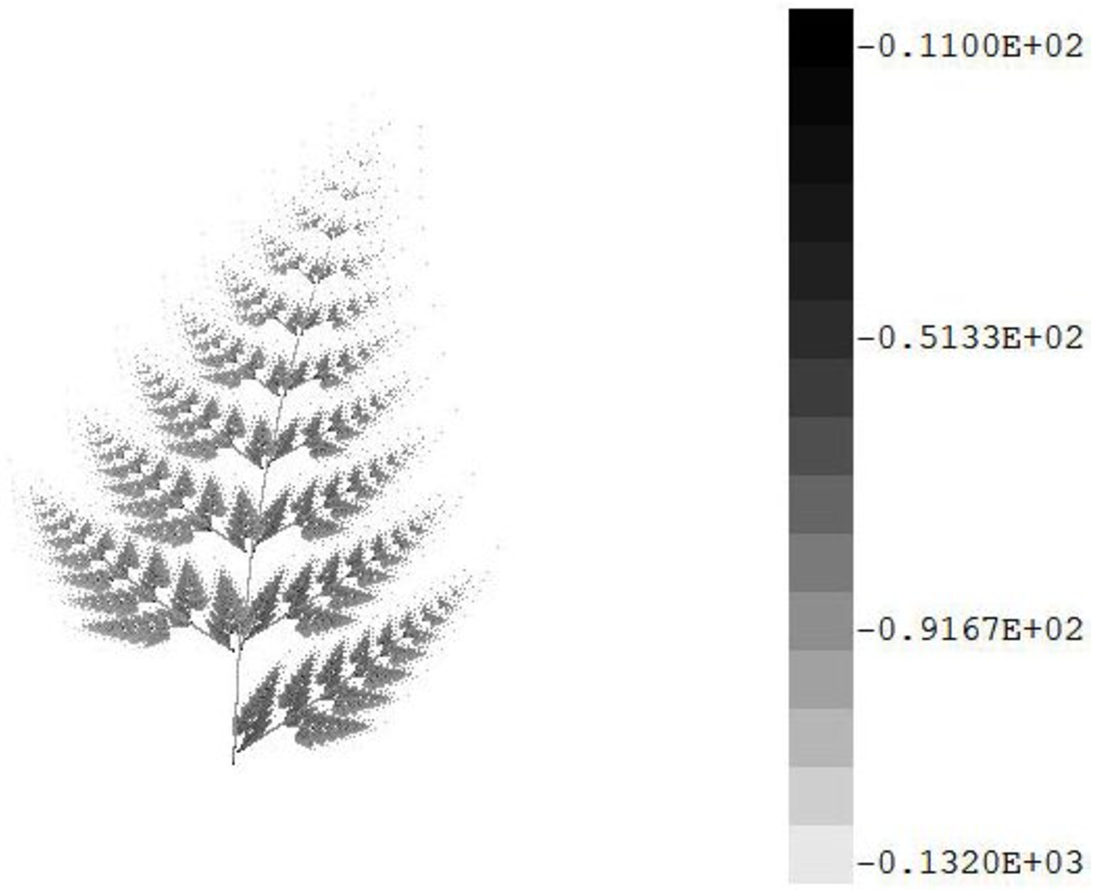}
  \caption{Output of algorithm \texttt{{ { DeterminIFSIdempMeasureDraw}}($\mS$)} after 12 iterations.}\label{Barnsley Fern}
  \end{figure}
\end{example}
In the next sections we will present alternative algorithm that gives more control on the number of points that have to be calculated. It bases on the idea on approximating IFSs by discrete ones, defined on appropriately dense grids. See \cite{COS20} and \cite{COS21}.
}
\subsection{Discretization of fuzzy IFSs}
\begin{definition}\emph{
A subset $\hat{X}$ of a metric space $(X,d)$ is called an }$\varepsilon$-net of $X$\emph{, if for every $x\in X$, there exists $y\in\hat{X}$ so that $d(x,y)\leq\varepsilon$. An $\varepsilon$-net $\hat{X}$ of $X$ is called }proper\emph{, if for every bounded set $D\subset X$, the set $D\cap\hat{X}$ is finite.\\
A map $r:X\to\hat{X}$ so that $r(x)=x$ for $x\in\hat{X}$ and $d(x,r(x))\leq\varepsilon$ for all $x\in X$ is called an }$\varepsilon$-projection\emph{ of $X$ to $\hat{X}$.\\
For $\phi:X\to X$, by its }$r$-discretization, \emph{we will call the map $\hat{\phi}:=(r\circ \phi)_{\vert \hat{X}}$.}
\end{definition}
Now we formulate \cite[Theorem 6.3]{COS21}, which is the key point for algorithms presented in that paper. Note that if {$\fS=(X,(\phi_j)_{j=1}^L,(d_j)_{j=1}^L)$ is a fuzzy IFS}, $\hat{X}$ is an $\varepsilon$-net of $X$ and $r$ is an $\varepsilon$-projection, then we can consider the {fuzzy} IFS $\hat{{\fS}}=(X,(\hat{\phi_j})_{j=1}^L,(d_j)_{j=1}^L)$ consisting of discretizations of $\phi_j$. The result below show that iterations of fuzzy Hutchinson operator of $Z_{\hat{\mS}}$ can somehow approximate the fuzzy attractor of $\mS$.\\
Below, for a fuzzy set $u\in {\F(X)}$, we set $e(u)$ to be the natural extension of $u$ to ${\F(X)}$, that is,
$$
e(u)(x):=\left\{\begin{array}{ccc}u(x)&\mbox{if}&x\in X{;}\\0&\mbox{if}&x\in X\setminus\hat{X}{.}\end{array}\right.
$$
Note that if $\hat{X}$ is proper, then $e(u)\in{\FK(X)}$ for all $u\in{\FK(\hat{X})}$
\begin{theorem}\label{fil20}
Let $(X,d)$ be a complete metric space and {${\fS}=(X, (\phi_j)_{j=1}^{L}, (d_{j})_{j=1}^{L})$} be a \textbf{Banach contractive} {fuzzy IFS} on $X$. Let $\ve>0$, ${\hat{X}}$ be a proper $\ve$-net, $r:X\to {\hat{X}}$ be an $\ve$-projection on ${\hat{X}}$ and $\hat{{\fS}}=(X, (\hat{\phi}_j)_{j=1}^L, (d_{j})_{j=1}^L)$, where $\hat{\phi}_j:=(r\circ \phi_{j})_{\vert {\hat{X}}}$ is the discretization of $\phi_j$.\\
Then for any $u\in{{\FK(\hat{X})}}$ and $n\in\N$,
{$${d_{\on{f}}}\Big(e\Big({Z}_{\hat{\mS}}^{(n)}(u)\Big),u_\mS\Big)\leq\frac{\ve}{1-\alpha_\mS}+\alpha_\mS^n{\;d_{\on{f}}}(e(u),u_\mS)$$}
where $u_\mS$ is the fuzzy attractor of $\mS$ and $\alpha_\mS:=\max\{\on{Lip}(\phi_j):j=1,...,L\}$. %{[FILIP: I erased the last part]}\ero{Elismar: OK}%\\
%In particular, there is $n_0\in\N$ such that for every $n\geq n_0$, ${d_{\on{f}}}(e({Z}_{\hat{\mS}}^{(n)}(u)),u_\mS)<\frac{2\ve}{1-\varepsilon}$.
\end{theorem}
We are going to formulate similar result for discretization of max-plus {normalized} IFSs. We could prove it directly, but we will make use Theorem \ref{fil7} and deduce it from the above Theorem \ref{fil20}.
\subsection{Discretization of idempotent measures}
For a compact space $(X,d)$ and proper (hence necessarily finite) $\varepsilon$-net and an idempotent measure $\mu\in I(\hat{X})$ with the density $\lambda:\hat{X}\to [-\infty,0]$, we define the idempotent measure $e(\mu)$ whose density $e(\lambda):X\to [-\infty,0]$ is a natural extension of $\lambda$, that is
$$
e(\lambda)(x):=\left\{\begin{array}{ccc}\lambda(x)&\mbox{if}&x\in \hat{X}{;}\\-\infty&\mbox{if}&x\in X\setminus\hat{X}{.}\end{array}\right.
$$
Clearly, $e(\mu)\in I(X)$ for every $\mu\in I(X)$.
\begin{theorem}\label{fil21}
{Let} {${\mpS}=(X, (\phi_j)_{j=1}^{L}, (q_{j})_{j=1}^{L})$} be a \textbf{Banach contractive} max-plus {normalized} IFS on $X$. Let $\ve>0$, ${\hat{X}}$ be a proper $\ve$-net, $r:X\to {\hat{X}}$ be an $\ve$-projection on ${\hat{X}}$ and ${\hat{\mS}_{\on{mp}}}=(X, (\hat{\phi}_j)_{j=1}^L, (d_{j})_{j=1}^L)$, where $\hat{\phi}_j:=(r\circ \phi_{j})_{\vert {\hat{X}}}$ is the discretization of $\phi_j$.\\
Then for any $\mu\in I(\hat{X})$ and $n\in\N$,
{$$d_\theta\Big(e\Big(M^{(n)}_{\hat{\mS}}(\mu)\Big),\mu_\mS\Big)\leq\frac{\ve}{1-\alpha_\mS}+\alpha_\mS^n{\;d_{\theta}}(e(\mu),\mu_\mS)$$}
where $\mu_\mS$ is the idempotent invariant measure of $\mS$ and $\alpha_\mS:=\max\{\on{Lip}(\phi_j):j=1,...,L\}$.
%{[FILIP: I also erased the last sentence]}\ero{Elismar: OK}
%\\
%In particular, there is $n_0\in\N$ such that for every $n\geq n_0$, $d_\theta(e(M^{(n)}_{\hat{\mS}}(\mu)),\mu_\mS)<\frac{2\ve}{1-\varepsilon}$.
\end{theorem}
\begin{proof}%{[FILIP: I shortened a bit the proof]}\ero{Elismar: OK}
{Take any $u\in {\FK(\hat{X})}$. It is easy to see that the density of $e(\Theta^{-1}(u))$ equals}
$$
e(\theta^{-1}\circ u)(x)=\left\{\begin{array}{ccc}\theta^{-1}\circ u(x)&\mbox{if}&x\in \hat{X}\\-\infty&\mbox{if}&x\in X\setminus\hat{X}\end{array}\right.{=}\theta^{-1}(e(u)(x))={(\theta^{-1}\circ e(u))}(x){.}
$$
Hence
$$
\Theta(e(\Theta^{-1}(u)))=\theta\circ \theta^{-1}\circ e(u)=e(u).
$$
 Now by Theorems \ref{fil13} and \ref{fil20}, and the {fact that the correspondence of the max-plus idempotent IFS $\hat{\mS}_{\on{mp}}$ is exactly the discretization $\hat{\fS}$ of $\fS$}, we have
$$
d_\theta(e(M^{(n)}_{\hat{\mS}}(\mu)),\mu_\mS){=}
{d_{\on{f}}}\Big(\Theta\Big(e\Big(\Theta^{-1}\Big(Z^{(n)}_{\hat{\mS}}(\Theta(\mu))\Big)\Big)\Big),\Theta(\mu_\mS)\Big)=
$$
$$
={d_{\on{f}}}\Big(e\Big(Z^{(n)}_{\hat{\mS}}(\Theta(\mu))\Big),u_{\hat{\mS}}\Big)\leq\frac{\varepsilon}{1-\alpha_\mS}+\alpha_{\mS}^n{d_{\on{f}}}(e(\Theta(\mu)),u_{\mS})=
$$
$$
=\frac{\varepsilon}{1-\alpha_\mS}+\alpha_{\mS}^nd_\theta(\Theta^{-1}(e(\Theta(\mu))),\Theta^{-1}(u_{\mS}))=\frac{\varepsilon}{1-\alpha_\mS}+\alpha_{\mS}^nd_\theta(e(\mu),\mu_{\mS})
$$
and the result follows.
\end{proof}

%of the above theorem and Theorem \ref{

\subsection{Idempotent algorithm}\label{subsec:Idempotent algorithm}
For a max plus normalized {IFS} ${\mpS}=(X,(\phi_j)_{j=1}^L,(q_j)_{j=1}^L)$, on a {compact} metric space consisting of Banach contractions, we present an algorithm to generate discrete invariant idempotent measure for ${\mpS}$ with a desired resolution $\delta$, as well an attractor of $\mS$ with resolution $\delta$. The presented algorithm is an adaptation of the Discrete deterministic algorithm for fuzzy IFS \texttt{FuzzyIFSDraw($\mS$)} from \cite{COS21} and it {bases} on Theorem \ref{fil21} and the description of the density of $M_\mS(\mu)$ from Lemma \ref{fil23}. This procedure can be summarized in the next algorithm.\\

{\tt
\begin{tabbing}
aaa\=aaa\=aaa\=aaa\=aaa\=aaa\=aaa\=aaa\= \kill
     \> \texttt{{ DiscreteIFSIdempMeasureDraw}(${\mS}$)}\\
     \> {\bf input}: \\
     \> \>  \> $\delta>0$, {a} resolution.\\
     \> \>  \> {$\hat{X}$ appropriately chosen proper net.}\\
     \> \>  \> $K \subseteq \hat{X}$, any finite and {nonempty} subset (a list of points in $\hat{X}$).\\
     \> \>  \> $\nu$, any discrete idempotent measure such that $\supp (\nu)= K$. \\
     \> \>  \> The diameter $D$ of a ball in $(X,d)$ containing {$K$ and the attractor $A_{\mS}$}.\\
     \> {\bf output:} \\
     \> \> \>  A bitmap representing a discrete attractor with resolution  at most $\delta$.\\
     \> \> \> A bitmap  image representing a discrete invariant idempotent measure,\\ \>\>\>having resolution  at most $\delta$,
 with a gray scale color bar indicating\\ \>\>\> the mass of each pixel.\\
     \> {\bf Compute:}\\
     \> \> \> $\displaystyle \alpha_\mS$ the Lipschitz constant of the {underlying IFS $\mS$}\\
     \> \> \> $\varepsilon >0$ and $N \in \mathbb{N}$ such that $\frac{\ve}{1-\alpha_\mS}+\alpha_\mS^N \, D < \delta$\\
     \> {\bf Initialize}  $\displaystyle\mu:={\bigoplus_{x \in \hat{X}}{-\infty\odot\delta_{x}}}$,  $W:=\varnothing$\\
     \> {\bf for n from 1 to N do}\\
     \>  \>{\bf for $\ell$ from 1 to $\operatorname{Card}({K})$ do}\\
     \>  \>  \>  \>{\bf for $j$ from $1$ to $L$ do}\\
     \>  \>  \>  \>  \>${x}:={K}[\ell]$\\
     \>  \>  \>  \>  \>${y}:=\hat{\phi}_{j}({x})$\\
     \>  \>  \>  \>  \>${W:=W\cup\{y\}}$\\
     \>  \>  \>  \>  \> {$\mu({y}):=\max\{\mu({y}) ,  q_{j} + \nu({x})\}$}\\
     \>  \>  \>  \>{\bf end do} \\
     \>  \>{\bf end do} \\
    % \> \> {$K:=W$}\\% and $W:=\varnothing$ \\
     \> \> $\nu:=\mu$, $\displaystyle\mu:=\bigoplus_{x \in \hat{X}}-\infty\odot\delta_{x}$, {$K:=W$ and $W:=\varnothing$}\\
     \>{\bf end do}\\
     \>{\bf return: Print} {$K$} and $\nu$ \\
\end{tabbing}}
%{[FILIP: I made small changes, in the algorithm, but we can go back to earlier version]}\ero{Elismar: OK, I'm happy with the present form}
\begin{remark}\label{remarkAAA}\emph{
  Actually, we could obtain our approximation of $\mu_{\mS}$ by first approximating $u_{\mS}$ via \texttt{FuzzyIFSDraw($\mS$)} and then {use} $\mu_{\mS}=\Theta^{-1}(u_{\mS})$. However the adapted  version we presented here, \texttt{{ DiscreteIFSIdempMeasureDraw}($\mS$)}, is {much faster} because we do not need the discretization of the extension principle. This is why we rewrite the algorithm \texttt{FuzzyIFSDraw($\mS$)} as \texttt{{ DiscreteIFSIdempMeasureDraw}($\mS$)}. The reciprocal {of this approach} will not work because not all fuzzy {IFSs are associated} to {max-plus normalized IFSs}.}
\end{remark}

\begin{example}\label{ex:cantor_Mazurenco}
This example from \cite{MZ} {(\cite[Example 1]{MZ})}, uses a classic geometric fractal, the Middle Third Cantor set. Consider $X=[0,1]$ and the {max-plus normalized IFS $\mpS=(X,(\phi_1, \phi_2),(q_1,q_2))$, where}%: X \to X$  with weights $(q_j)_{j=1}^{L=2}$ where
{
\[
\left\{
  \begin{array}{ll}
      \phi_1(x)   & = \frac{1}{3}x\\
      \phi_2(x)   & = \frac{1}{3}(x+2)\\
  \end{array}
\right.
\]
}
and $q_1= 0, q_2= -1$. From \cite{MZ} we know that the invariant idempotent measure is:
$$
\mu_{\mS}=0 \odot \delta_{0} \oplus \bigoplus_{1 \leq i_{1}<\cdots<i_{k}}(-k) \odot \delta\left(\sum_{j=1}^{k} \frac{2}{3^{i_{j}}}\right).
$$
Consider, for example, the continuous function ${\vp}(x)=-2+3 x$. As $-2 \leq {\vp} \leq 1$ if {$$\mu_{\mS}({\vp})=(-k) \odot {\vp}\left(\sum_{j=1}^{k} \frac{2}{3^{i_{j}}}\right)$$} we get $-2 \leq  -k + 1$ or $k \leq 3$. Testing all the values in this range we conclude that $\mu_{\mS}({\vp})=(-1) \odot {\vp}{\left(\frac{2}{3}\right)} =-1$.\\

In the general case we do not have the exact formula for $\mu_{\mS}{(\vp)}$. {However, we can} use the algorithm \texttt{{ DiscreteIFSIdempMeasureDraw}($\mS$)}. For {example, for} a discretization with 100 points and 15 iterations we obtain $\mu_{\mS}({\vp})=-{0.9900}$ in few seconds and for 1000 points we get $\mu_{\mS}({\vp})=-{0.999000}$ and so on.\\

The picture of $\mu_{\mS}$ produced by the algorithm \texttt{{ DiscreteIFSIdempMeasureDraw}($\mS$)} do not show much graphically because its density is equal to $-\infty$ outside of {the Cantor ser}. However, the correspondent fuzzy set{, that is,} $$u_{\mS}=\Theta(\mu_{\mS}),$$
for $\theta(t)=1.1^t, t \leq 0$, given { on the left side of} Figure~\ref{fig:cantor_by_theta}, {is} indistinguishable from the right one, which was produced by the algorithm \texttt{FuzzyIFSDraw($\mS$)} (see \cite{COS21}), for the associated fuzzy IFS with {admissible system of grey level maps}\\
$$r_{1}(t)=\theta(q_1 + \theta^{-1}(t))= t \text{ and } r_{2}(t)=\theta(q_2 + \theta^{-1}(t))=  1.1^{-1+ \frac{\ln t}{\ln(1.1)}},$$
as predicted by Theorem~\ref{fil7}. As both approaches are equivalent, starting with the computation of $u_{\mS}$ by using the algorithm \texttt{FuzzyIFSDraw($\mS$)}, with a discretization with 1000 points and 15 iterations, we obtain again $\mu_{\mS}({\vp})=-0.999000$ for $\mu_{\mS}=\Theta^{-1}(u_{\mS})$.\\
\begin{figure}[H]
  \centering
  \includegraphics[width=4cm]{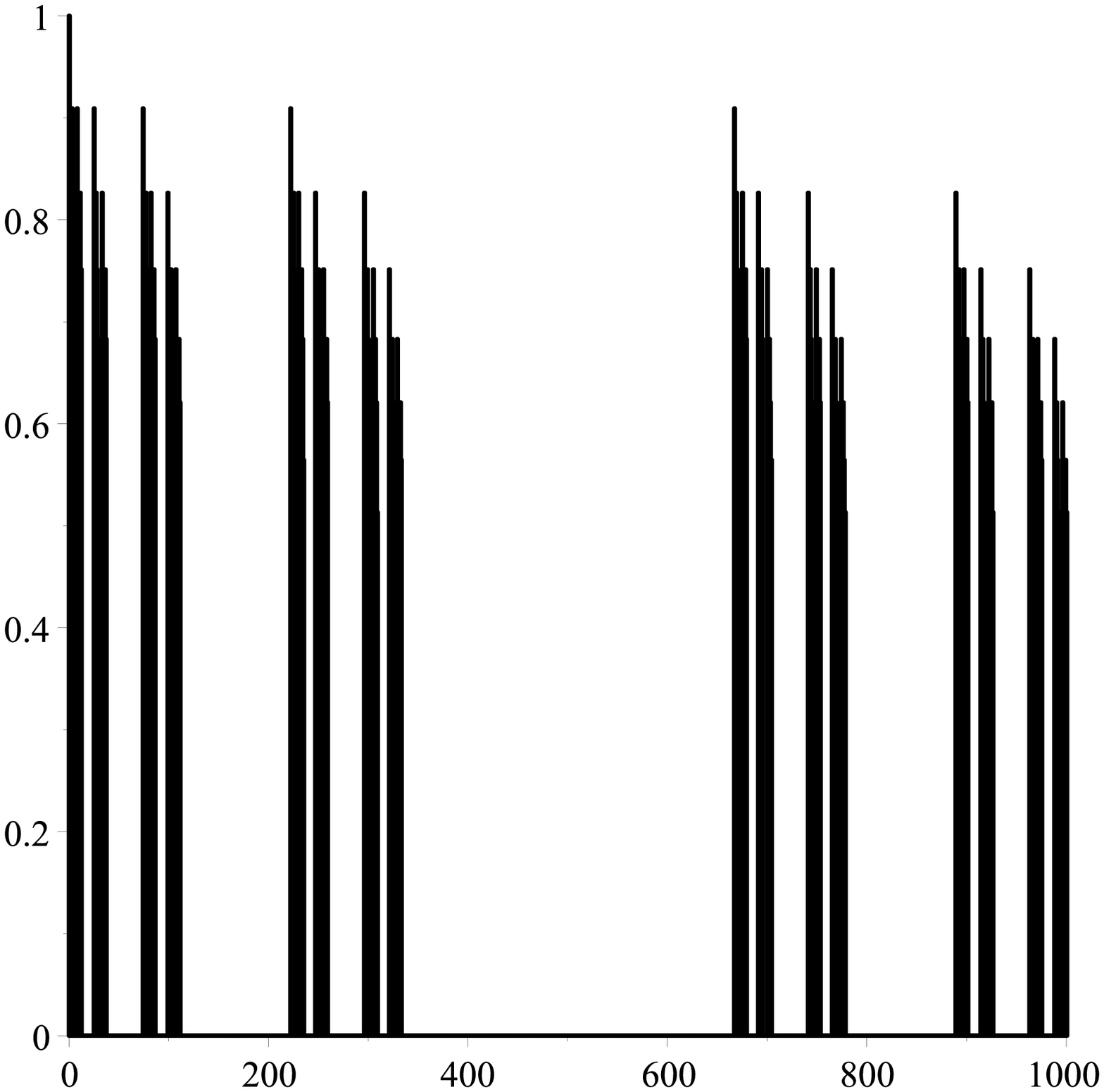}\;\quad
  \includegraphics[width=4cm]{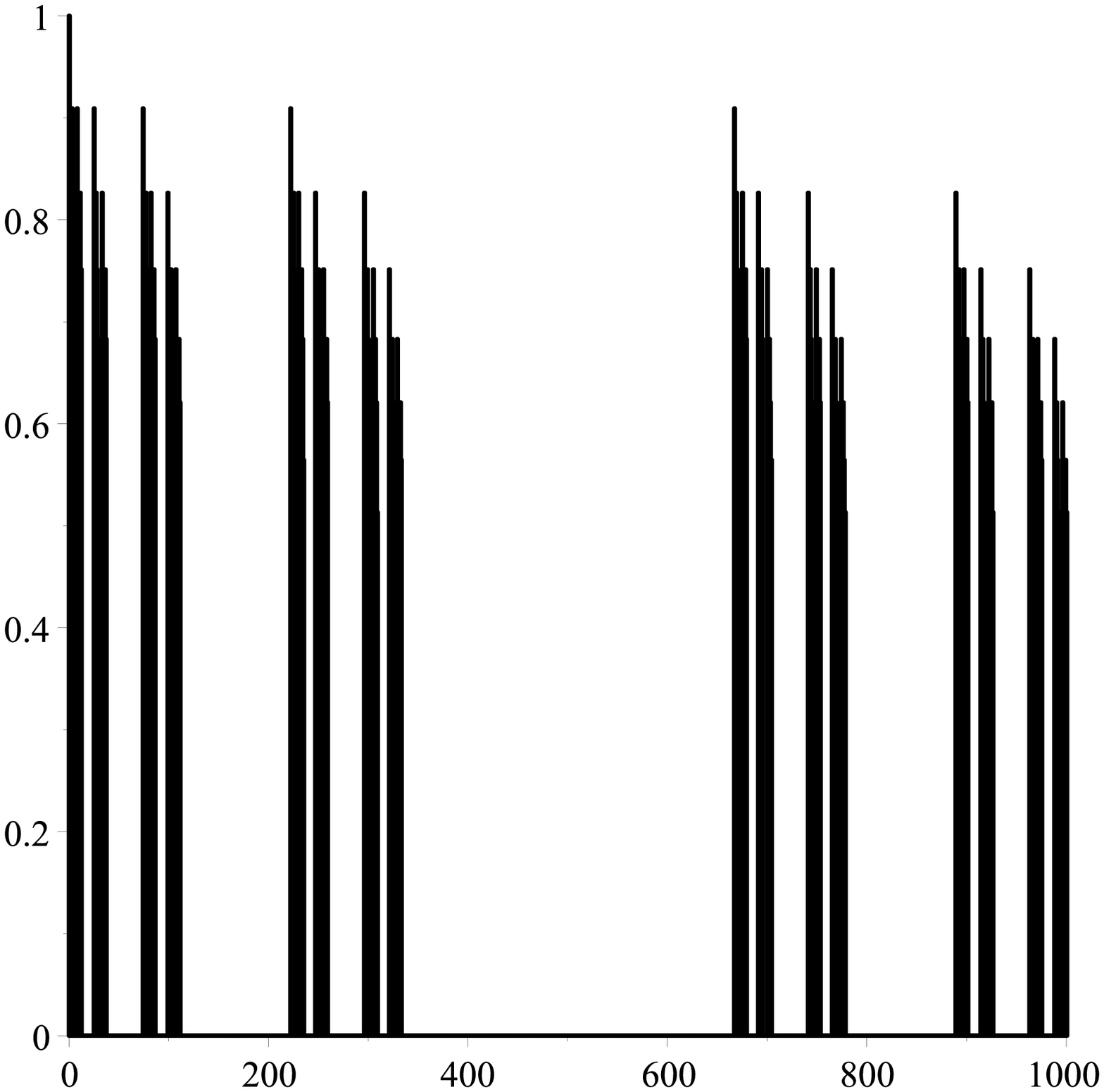}
  \caption{{On} the left, the output of the algorithm \texttt{{ DiscreteIFSIdempMeasureDraw}($\mS$)}  transformed by $\Theta$ and, {on} the right, the fuzzy attractor of the associated fuzzy IFS obtained by the algorithm \texttt{FuzzyIFSDraw($\mS$)}{.}}\label{fig:cantor_by_theta}
\end{figure}
\end{example}

\begin{example}\label{ex:maple_idemp_measures} This example uses a classic geometric fractal, the Maple Leaf. The approximation of the discrete idempotent invariant measure, through \texttt{{ DiscreteIFSIdempMeasureDraw}($\mS$)}, is presented {on} the Figure~\ref{MapleIdempMeasure}.  Consider $X=[0,1]^2$ and the {max-plus normalized IFS $\mpS=(X,(\phi_j)_{j=1}^4,(q_j)_{j=1}^4)$, where}%  with weights $(q_j)_{j=1}^{L=4}$ where
\[
\left\{
  \begin{array}{ll}
      \phi_1(x,y)   & = (0.8 x + 0.1, 0.8 y + 0.04)\\
      \phi_2(x,y)   & = (0.5 x + 0.25, 0.5 y + 0.4)\\
      \phi_3(x,y)   & = (0.355 x - 0.355 y +0.266,  0.355 x + 0.355 y + 0.078)\\
      \phi_4(x,y)   & = (0.355 x + 0.355 y +0.378,  -0.355 x + 0.355 y + 0.434)\\
  \end{array}
\right.
\]
and $q_1= 0, q_2= -7, q_3= -3$ and $q_4=-7$.\\
As we can see, in Figure~\ref{MapleIdempMeasure}, the darkness of each point $x \in \hat{X}$ represents how close to zero the density ${\lambda_\mS}$ of ${{\mu_{\mathcal{S}}}}$ is. The white points has density equal to $-\infty$.
\begin{figure}[H]
  \centering
  \includegraphics[width=8.5cm]{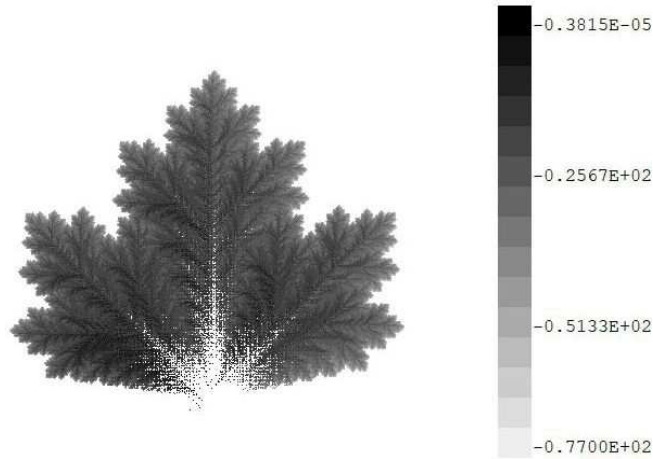}\;
  \caption{Algorithm \texttt{{ DiscreteIFSIdempMeasureDraw}($\mS$)} after 11 iterations in a $512\times 512$ net, producing only 206904 at the support from an initial one $\displaystyle\nu:=\delta_{(0.5,0.5)}$.}\label{MapleIdempMeasure}
  \end{figure}
In order to make a comparison we also run the algorithm \texttt{{ DeterminIFSIdempMeasureDraw}($\mS$)}, whose output is displayed in Figure~\ref{DiscreteMapleIdempMeasure}. We notice that Figure~\ref{MapleIdempMeasure} is much well defined and darker because we were able to perform much more iterations and the color scale is proportional to the lowest value at the support.
\begin{figure}[H]
  \centering
  \includegraphics[width=9cm]{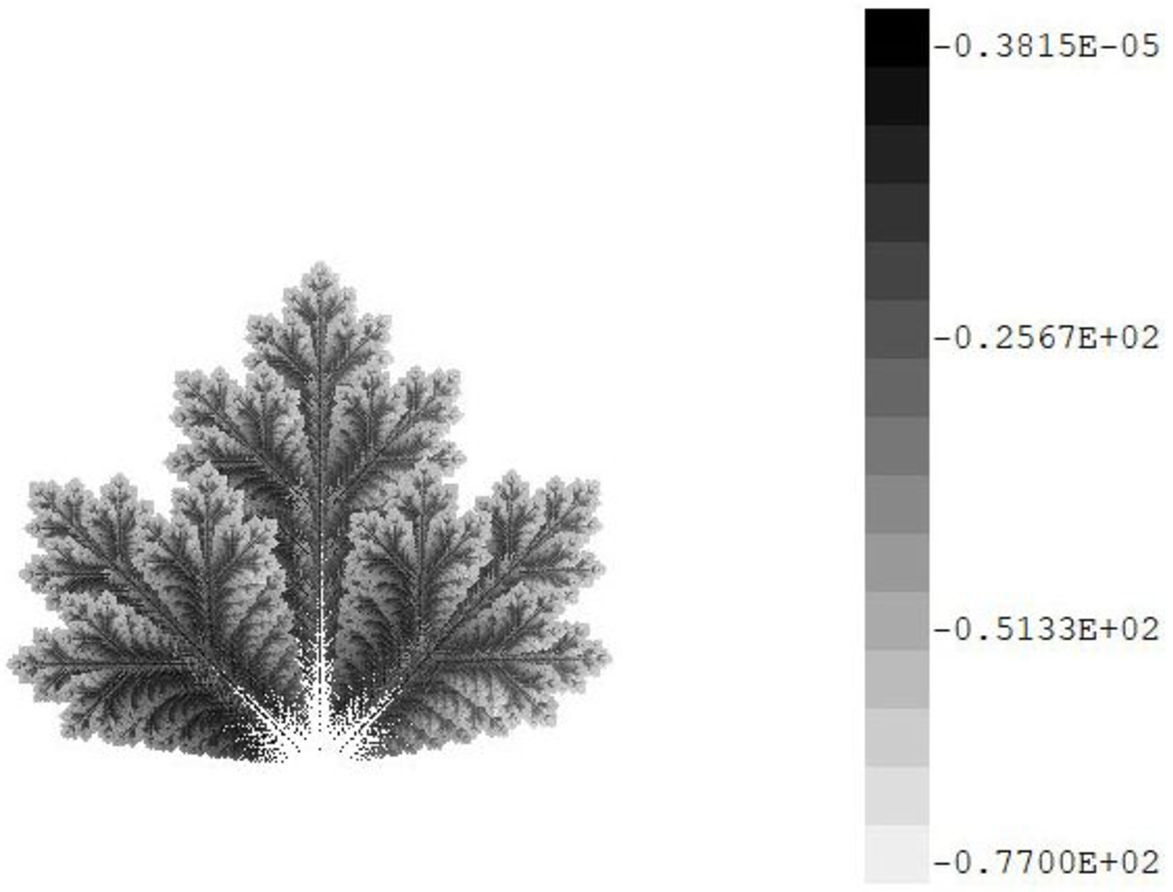}\;
  \caption{Algorithm \texttt{{ DeterminIFSIdempMeasureDraw}($\mS$)} after 12 iterations producing 4194304 points from an initial one $\displaystyle\nu:=\delta_{(0.5,0.5)}$.}\label{DiscreteMapleIdempMeasure}
  \end{figure}
\end{example}

\section{Algorithms for max-plus generalized IFSs in the sense of Miculescu and Mihail}
As we mentioned in Remark \ref{remarknew}, our results can be rewritten to the setting of generalized IFSs (GIFSs) in the sense of Miculescu and Mihail. For brief information on GIFSs we refer the reader to Miculescu and Mihail papers, for example \cite{Mih08} and \cite{Mih10}. Not going into details, let us point out instead of selfmaps of a metric space $X$, GIFSs consist of maps $\phi_i$ defined on finite product $X^m$ and with values in $X$. It turns out that a great part of the classical theory has a natural counterpart in such setting. In particular, contractive GIFSs $\mS$ generates attractors $A_\mS\in\K(X)$ that satisfy
$$
A_\mS=\bigcup_{i=1}^L\phi_i(A_\mS\times...\times A_\mS).
$$
In \cite{Oli17} we considered a fuzzy version of GIFSs, and in \cite{COS21} we introduced the discrete algorithm for generating images of fuzzy GIFSs attractors. It is easy to rewrite algorithms \texttt{DiscreteIFSIdempMeasureDraw}($\mS$)} and \texttt{{{DeterminIFSIdempMeasureDraw}}(${\mS}$)} for \emph{max plus normalized} GIFSs (whose definition is analogous to the definition of max plus normalized IFSs, and is clear from the context).\\
The only difference is that, instead of defining
$$
x:=K[\ell]
$$
we have to define
$$
x:=[K[\ell_1],...,K[\ell_m]]
$$
where $\ell_1,...,\ell_m$ are all taken from $1$ to $\on{Card}(K)$.

{Consider \texttt{{ DiscreteGIFSIdempMeasureDraw}(${\mS}$)} as being the algorithm obtained from\newline \texttt{{ DiscreteIFSIdempMeasureDraw}(${\mS}$)}  by the above modification.  We now present some examples for generalized IFSs, only for the discrete version due to its efficiency.
\begin{example}\label{example filip measures} This example uses the GIFS $\mathcal{G}$ appearing in \cite[Example 16]{Jaros-2016}. The approximation of the discrete idempotent invariant measure, through \texttt{{ DiscreteGIFSIdempMeasureDraw}($\mS$)}, is presented in Figure~\ref{FilipMeasurePic}.  Consider $X=[0,1]^2$ and the max-plus normalized IFS $\mpS=(X,(\phi_j)_{j=1}^3,(q_j)_{j=1}^3)$, where
\[\mS:
\left\{
  \begin{array}{ll}
       \phi_1((x_1,y_1),(x_2,y_2))&=(0.25 x_1+0.2 y_2, 0.25 y_1+0.2 y_2) \\
       \phi_2((x_1,y_1),(x_2,y_2))&=(0.25 x_1+0.2 x_2, 0.25 y_1+0.1 y_2+0.5)\\
       \phi_3((x_1,y_1),(x_2,y_2))&=(0.25 x_1+0.1 x_2+0.5, 0.25 y_1+0.2 y_2))\\
  \end{array}
\right.
\]
and $q_{1}=-2, \; q_{2}=0, \;q_{3}= 0$.
%Figure~\ref{FilipMeasurePic} shows the approximation obtained by \texttt{{ DiscreteGIFSIdempMeasureDraw}($\mS$)}.
 \begin{figure}[H]
  \centering
  \includegraphics[width=8cm]{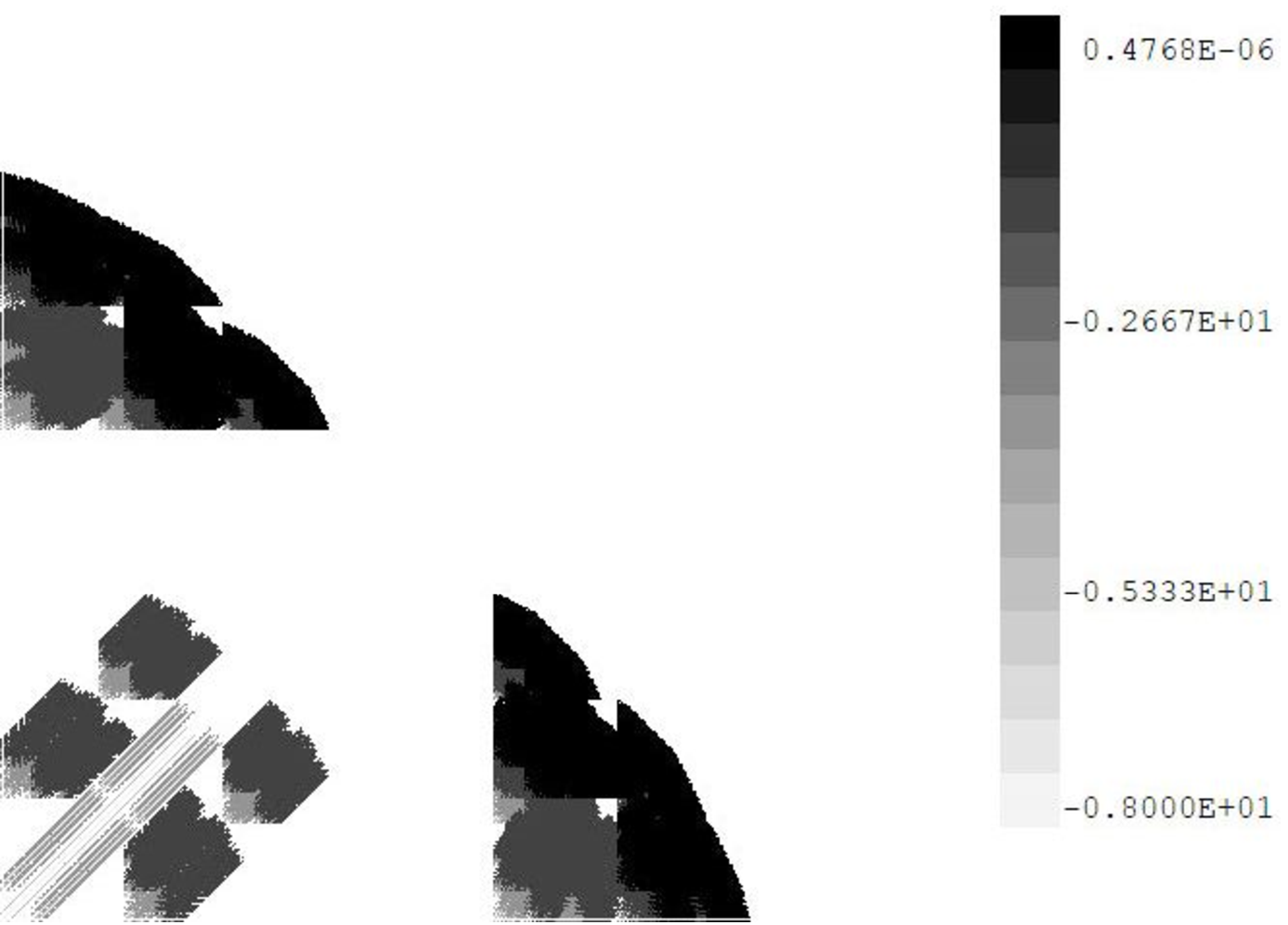}
  \caption{The idempotent invariant measure obtained by \texttt{{ DiscreteGIFSIdempMeasureDraw}($\mS$)} using a $512\times512$ pixels and 4 iterations.}\label{FilipMeasurePic}
  \end{figure}
\end{example}
\begin{example}\label{example final}
 This example came from \cite[Example 11.6]{COS21}. The approximation of the discrete idempotent invariant measure, through \texttt{{ DiscreteGIFSIdempMeasureDraw}($\mS$)}, is presented in Figure~\ref{example_final}.  Consider $X=[-0.1,2.1]^2$  and the max-plus normalized IFS $\mpS=(X,(\phi_j)_{j=1}^3,(q_j)_{j=1}^3)$, where
\[\mS :\left\{
  \begin{array}{ll}
    \phi_1((x_1,y_1),(x_2,y_2))&=(0.2 x_1+0.25 x_2+0.04 y_ 2,0.16 y_1-0.14 x_2+0.20 y_2+1.3) \\
    \phi_2((x_1,y_1),(x_2,y_2))&=(0.2 x_1-0.15 y_1-0.21 x_2+0.15 y_2+1.3,0.25 x_1+0.15 y_1+0.25 x_2+0.17) \\
    \phi_3((x_1,y_1),(x_2,y_2))&=(0.355  x_1+0.355  y_ 1+0.378,-0.355 x_1+0.355  y_1+0.434-0.03  y_2)\\
  \end{array}
\right.
\]
and $q_{1}=-1, \; q_{2}=0, \;q_{3}= -7$.
%Figure~\ref{example_final} shows the approximation obtained by \texttt{{ DiscreteGIFSIdempMeasureDraw}($\mS$)}.
\begin{figure}[H]
\centering
  \includegraphics[width=8cm]{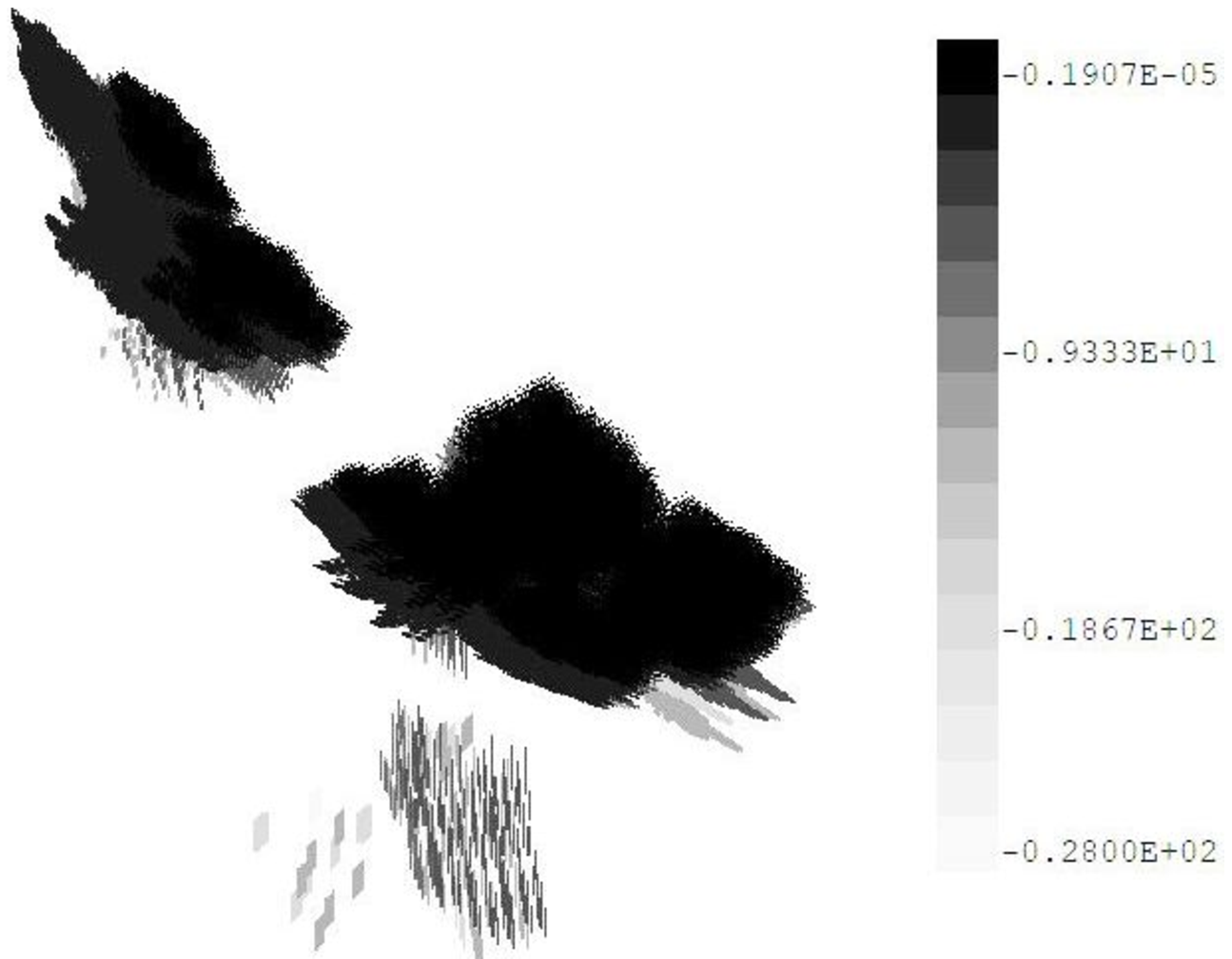}
  \caption{The idempotent invariant measure obtained by \texttt{{ DiscreteGIFSIdempMeasureDraw}($\mS$)} using a $512\times512$ pixels and 4 iterations.}\label{example_final}
  \end{figure}
\end{example}
}

\end{document}